\documentclass{article}
\usepackage[utf8]{inputenc}
\usepackage{amsmath, amscd}
\usepackage{amssymb, bm} 
\usepackage{amsthm}
\usepackage{color}
\usepackage[dvipsnames]{xcolor}
\usepackage{cancel}
\usepackage{ulem}
\usepackage{geometry}
\usepackage[all]{xy}
\usepackage{mathrsfs}
\usepackage{mathtools}
\usepackage{tikz}
\usepackage{multicol}
\usepackage{soul}
\usepackage{graphicx, psfrag, tikz, tikz-cd}

\usepackage{caption}
\usepackage{subcaption}

\usepackage{paralist}

\numberwithin{equation}{section}

\newtheorem{theorem}{Theorem}[section]
\newtheorem{lemma}[theorem]{Lemma} 
\newtheorem{proposition}[theorem]{Proposition}

\theoremstyle{definition} 
\newtheorem{definition}[theorem]{Definition}

\newtheorem{remark}[theorem]{Remark}
\newtheorem{example}[theorem]{Example}

\DeclareMathOperator{\Sp}{Sp}

\DeclareMathOperator{\Hom}{Hom}

\DeclareMathOperator{\Coh}{Coh}
\DeclareMathOperator{\Fuk}{Fuk}

\newcommand{\bC}{\mathbb C}
\newcommand{\bD}{\mathbb D}

\newcommand{\bR}{\mathbb R}

\newcommand{\bZ}{\mathbb Z}

\newcommand{\bj}{\mathsf{j}}

\newcommand{\cD}{\mathcal D}
\newcommand{\cE}{\mathcal E}

\newcommand{\cL}{\mathcal L}
\newcommand{\cM}{\mathcal M}

\newcommand{\cW}{\mathcal W}

\newcommand{\fM}{\mathfrak{M}}
\newcommand{\fT}{\mathfrak{T}}

\newcommand{\Hol}{\mathrm{Hol}}

\newcommand{\pr}{\mathrm{pr}}

\newcommand{\tL}{\widetilde{L}}

\newcommand{\uX}{\underline{X}}

\newcommand{\hL}{\hat{L}}

\newcommand{\dd}{\partial}
\newcommand{\pf}{\partial f}
\newcommand{\pS}{\partial S}

\newcommand{\lra}{\longrightarrow} 

\definecolor{Cgreen}{RGB}{77,175,74}
\definecolor{Cblue}{RGB}{0, 150, 255}
\definecolor{Corange}{RGB}{255,127,0}
\usepackage[colorlinks=true, linkcolor=blue, filecolor=blue,citecolor=blue]{hyperref}

\title{Fukaya category on a symplectic manifold with a B-field}
\author{Haniya Azam, Catherine Cannizzo, Heather Lee, Chiu-Chu Melissa Liu}
\date{}

\begin{document}

\maketitle

\noindent

\begin{abstract}
We describe the formulation of  Fukaya categories of symplectic manifolds with B-fields.  In addition, we give a formula for  how the $A_\infty$ structure maps change as we deform an object by a Lagrangian isotopy.
\end{abstract}

\tableofcontents

\vspace{12pt}

\section{Introduction}

In this paper, we consider the Fukaya category on a symplectic manifold with a B-field  in the following sense. 

\begin{definition}A {\em B-field} on a symplectic manifold $(X,\omega)$ is a real closed 2-form 
$B\in \Omega^2(X)$.   Given a symplectic manifold $(X,\omega)$ with a B-field $B$, we define
the {\em complexified symplectic form} to be the complex 2-form $\omega_\bC = B+ i\omega \in \Omega^2(X,\bC)$. 
\end{definition}

One motivation for introducing the B-field comes from mirror symmetry.  For example, suppose $X$ and $Y$ were a mirror pair of Calabi-Yau manifolds.  The mirror map gives a local isomorphism between the moduli space of complexified K\"{a}hler structures on $X$, which has complex dimension $h^{1,1}(X)$, and the moduli space of complex structures on $Y$, which has  complex dimension $h^{n-1,1}(Y)$ where $n$ is the complex dimension of $Y$,  hence  $h^{1,1}(X)=h^{n-1,1}(Y)$.  The space of 
K\"ahler classes on $X$ is of real dimension $h^{1,1}(X)$, so it is not enough to determine all the complex structures on the mirror.  At the level of homological  mirror symmetry (HMS) conjectured and formulated by Kontsevich \cite{hms}, the mirror statement is that 
\begin{equation}\label{eq:hms}
D^\pi\Fuk (X, (\omega_X)_\bC)\cong D^b\Coh (Y, J_Y),
\end{equation}
where $J$ denotes the complex structure,  $D^\pi \Fuk(X,(\omega_X)_\bC)$ is the split-closed derived Fukaya category,  and $D^b\Coh(Y, J_Y)$ is the bounded derived category of coherent sheaves.

The introduction of B-field is not new, as it originated from the study of strings in background fields \cite[Section 3.4]{GSW}, then it appeared in generalized complex geometry \cite{HitchenGeneralizedCY}, as well as a volume of (mostly early) work on SYZ mirror symmetry \cite{GrossSYZ, Hi99} and HMS \cite{zp, Fuk02, AKO} (apologies for this very incomplete list of references as it is impossible to list them all).   However, we cannot find any existing literature that details all the basic ingredients for constructing the Fukaya category on a general symplectic manifold when B-field is involved, so we decided to write it ourselves which resulted in this paper.  Many of the recent works on HMS prove  statements like Equation \eqref{eq:hms} for particular families of real symplectic forms and the mirror complex structures (note the HMS statement is a bit different from Equation \eqref{eq:hms} when the manifolds involved are not Calabi-Yau).  In our forthcoming work \cite{ACLLa}, we prove a global HMS result for all complex structures on the genus two curves and the complexified symplectic structures on its mirror.  Then it becomes important to understand the Fukaya category on a symplectic manifold with B-field.

Computing the $A_\infty$ product $\mu^k$ of the Fukaya category on $(X,\omega_\bC)$ is often difficult, and one useful strategy is to deform one (or more) of the Lagrangian objects to one(s) for which $\mu^k$ is easier to compute.  In Section \ref{sec: isotopy}, we deform an object by a Lagrangian isotopy (which is more general than a Hamiltonian isotopy) and give an explicit formula for how $\mu^k$ changes as we perform such a deformation.

\paragraph{Acknowledgement.}  C. Cannizzo was partially supported by NSF DMS-2316538.

\section{Objects}\label{sec:objects}

Let $(X,\omega_\bC = B + i\omega)$ be a symplectic manifold of real dimension $2n$ with a B-field, where $\omega$ is a symplectic
form and $B$ is a closed 2-form.  An object in  the Fukaya category
$\Fuk(X,\omega_\bC)$ of $(X,\omega_\bC)$ is a triple $\hat{L}= (L,\cL, \nabla)$, where
\begin{enumerate}
\item $L$ is a  Lagrangian submanifold in the symplectic manifold $(X,\omega)$, i.e., $\dim L = n$ and $\omega|_L =0$. 

\item $L$ is oriented and equipped with a spin structure\footnote{Recall that $L$ is orientable iff $w_1(TL)=0\in H^1(L;\bZ_2)$; in this case, all the possible orientations on $L$ form a {\em torsor} (i.e. principal homogeneous space) for $H^0(L;\bZ_2)$. If $L$ is orientable then it is
spin iff $w_2(TL)=0 \in H^2(L;\bZ_2)$; in this case all the possible spin structures form a torsor for $H^1(L;\bZ_2)$.}.

\item  $\cL$ is a $C^\infty$ complex line bundle on $L$.
\item $\nabla$ is a unitary connection on  $\cL$.
\item The curvature $F_\nabla$ of $\nabla$ satisfies $F_\nabla = - 2\pi i  B|_L$.
\end{enumerate} 

\begin{remark}
Recall that the first {\em Chern  form} of $(\cL, \nabla)$ is a real 2-form on $L$: 
$$
c_1(\cL, \nabla) =\frac{i}{2\pi} F_\nabla \in \Omega^2(L). 
$$
So condition 4 can be re-written as
$$
B|_L = c_1(\cL,\nabla).
$$
Therefore, 
$$
[B|_L]  = c_1(\cL) \in H^2_{\mathrm{dR}}(L;\bR) = H^2(L;\bR).
$$
\end{remark}

\begin{remark}
In \cite{Fuk02}, $\Omega = \omega + iB = -i(-B + i\omega)$, so $B$ in this paper corresponds to  $-B$ in  \cite{Fuk02}
and in particular \cite[Definition 1.1]{Fuk02}. 
\end{remark} 

Suppose that $\nabla_0, \nabla_1$ are two unitary connections on a $C^\infty$ complex line bundle
$\cL$ on a Lagrangian submanifold $L$ in $X$ such that
$$
F_{\nabla_j} = -2\pi i B|_L, \quad j=0,1.
$$
Then $\nabla_1 -\nabla_0 = -2\pi i a$ for some closed real 1-form $a\in \Omega^1(L)$.

\subsection{Hamiltonian equivalence} \label{subsec: product def Hamiltonian equivalence} 

Recall that a Hamiltonian isotopy on a symplectic manifold $(X,\omega)$ is a smooth map 
$\varphi: X\times [0,1]\to X$ such that 
$$
\varphi(x,0) =x, \quad \frac{\partial \varphi}{\partial t}(x,t) = X_H(\varphi(x,t)) 
$$ 
where $X_H$ is the Hamiltonian vector field associated to a time dependent Hamiltonian
$H: X\times [0,1] \to X$. For each $t\in [0,1]$, define $\varphi^t: X\to X$ by $\varphi^t(x)= \varphi(x,t)$.

\begin{definition} Let $\hat{L} = (L, \cL,  \nabla)$ and $\hat{L}' = (L', \cL',  \nabla')$ be two objects in $\Fuk(X,\omega_\bC)$.
We say $\hat{L}$ is {\em Hamiltonian equivalent} to  $\hat{L}'$ if there exists
a Hamiltonian isotopy $\varphi: X\times[0,1]\to X$, a $C^\infty$ complex line bundle $\tilde{\cL}$ on $L\times [0,1]$, and a unitary connection $\tilde{\nabla}$ on  $\tilde{\cL}$, such that
\begin{enumerate}
\item $\varphi^1(L) =L'$. The orientation and spin structure
on $L$ is the pullback of those on $L'$ under the diffeomorphism $\varphi^1: L\to L'$.
\item The curvature $F_{\tilde{\nabla}}$ of $\tilde{\nabla}$ satisfies $F_{\tilde{\nabla}} = -2\pi i  \left. (\varphi^*B)\right|_{L\times [0,1]}$.
\item Given $t\in [0,1]$, define $\iota_t: L\to L\times [0,1]$ by $\iota_t(x) = (x,t)$. Then
$$
\iota_0^*(\tilde{\cL}, \tilde{\nabla}) = (\cL, \nabla),\quad
\iota_1^*(\tilde{\cL},  \tilde{\nabla}) = (\varphi^1)^*(\cL',  \nabla').
$$
\end{enumerate} 
\end{definition}

Let $\fM(\Fuk(X,\omega_\bC))$ be the moduli space of Hamiltonian equivalence classes of objects in $\Fuk(X,\omega_\bC)$; this  moduli space 
is denoted $\mathcal{LAG}(M,\Omega)$ in \cite[Chapter 1 \S 1]{Fuk02}, where $(M,\Omega)$ corresponds to $(X,\omega_\bC)$ in this paper. Let $\hat{L} =(L,\cL, \nabla)$ be an object in $\Fuk(X,\omega_\bC)$, and let $[\hat{L}] \in \fM(\Fuk(X,\omega_\bC))$ denote its Hamiltonian equivalence class. Then 
$$
T_{[\hat{L}]} \fM(\Fuk(X,\omega_\bC)) \simeq  H^1(L;\bC). 
$$
We refer to \cite[Chapter 1 \S 1]{Fuk02} for more details.

\section{Morphisms} \label{sec: morphisms}
Let $\hat{L}_0 = (L_0,\cL_0,  \nabla_0)$ and  $\hat{L}_1 = (L_1,\cL_1,  \nabla_1)$ be two objects in $\Fuk(X,\omega_\bC)$, and 
assume that the Lagrangian submanifolds $L_0$ and $L_1$ intersect transversally at finitely many points. Then 
$$
\Hom_{\Fuk(X,\omega_\bC)} (\hat{L}_0, \hat{L}_1)  = CF(\hat{L}_0, \hat{L}_1).
$$ 
As a complex vector space,
$$
CF(\hat{L}_0, \hat{L}_1) = \bigoplus_{p\in L_0\cap L_1} \Hom_{\bC}( (\cL_0)_p,  (\cL_1)_p),
$$
where $\Hom_{\bC}( (\cL_0)_p,  (\cL_1)_p)\simeq \bC$ is the space of  $\bC$-linear maps from $(\cL_0)_p$ to $(\cL_1)_p$.  For any $p\in L_0\cap L_1$,
\begin{equation}\label{eqn:duality-p}
\Hom_{\bC} ((\cL_1)_p, (\cL_0)_p) \cong \Hom_{\bC} ((\cL_0)_p, (\cL_1)_p)^\vee.
\end{equation}
Therefore, 
\begin{equation}\label{eqn:duality-CF}
CF(\hat{L}_1, \hat{L}_0) = CF(\hat{L}_0, \hat{L}_1)^\vee.
\end{equation}

\section{Grading}\label{sec:grading}
The main reference of this subsection is P. Seidel's original paper \cite{seidel_lagr}.  We also follow \cite[Section 1.3]{fuk_intro} and \cite[Section 2.6]{math257b} at places.

\subsection{Preliminaries on principal bundles and associated fiber bundles}
Let $G$ be a Lie group, and let $\pi: P\to X$ be a principal $G$-bundle, i.e. a principal fiber bundle with total space $P$, base $X$, and structure group $G$. Then $G$ acts smoothly and freely on $P$ on the right, and $\pi: P\to X$ is the natural projection $P\to P/G$. 

Let $F$ be a smooth manifold equipped with a left $G$-action. Then  
$G$ acts freely on $P\times F$ on the right by 
$$
(u,\xi)\cdot g = (u\cdot g, g^{-1}\cdot \xi) \quad 
g\in G, \ u\in P, \ \xi\in F.
$$
Let $P\times_G F:= (P\times F)/G$ be the quotient by the above action. 
We have the following Cartesian diagram: 
$$
\xymatrix{
P\times F \ar[r] \ar[d]^{\pr_1} & P\times_{G} F = (P\times F)/G \ar[d]^{\hat{p}}\\
P\ar[r]^\pi &  X = P/G
}
$$
where $\pr_1$ is projection to the first factor, $\hat{p}: P\times_G F\to X$ is a smooth fiber bundle
with fiber $F$ and base $X$, and the horizonal arrows are principal $G$-bundles.

\begin{example} \label{ex:homogeneous-space}

Let $H$ be a closed subgroup of $G$. Then $H$ is a Lie subgroup, $G/H =\{ aH: a\in G\}$ is a smooth manifold, and the natural projection $G\to G/H$ is a principal $H$-bundle. $G$ acts on $G/H$ smoothly and transitively on 
the left by $g\cdot aH = gaH$ where $g, a\in G$. The associated fiber bundle
$P\times_{G} G/H$ can be identified with $P/H$, where the right $H$-action 
is the restriction of the free right $G$-action on $P$, and we have natural projections
$$
P\lra P/H = P\times_{G} G/H \lra  X = P/G, 
$$
where $P\to P/H$ a principal $H$-bundle, and $P/H\to P/G$ is a fiber bundle with base $X$ and fiber $G/H$.
\end{example}

\begin{example}[associated vector bundles] \label{ex:associated-vector-bundle}
Let $\rho: G\lra GL(r,\bC)$ be a group homomorphism. Then $G$ acts $\bC$-linearly on $\bC^r$ on the left by 
$g\cdot v = \rho(g) v$, where we view $v \in \bC^r$ as a column vector with $r$ components. The associated fiber bundle
$$
P\times_\rho \bC^r := P\times_G \bC^r
$$
is a complex vector bundle of rank $r$ over $X$, known as the vector bundle associated
to the principal $G$-bundle $P$ and the representation $\rho$. 
\end{example}

\begin{example}\label{ex:structure-group}
Let $\phi: G\to K$ be a homomorphism of Lie groups. Then $G$ acts on $K$ on the left by 
$$
g\cdot k = \phi(g) k,\quad g\in G,\ k\in K.
$$
The associated fiber bundle $P\times_G K$ is a principal $K$-bundle over $X$.
\end{example}

\subsection{Lagrangian Grassmannian}
As a set, the Lagrangian Grassmannian $LGr(\bR^{2n})$ is the set of all linear Lagrangian subspaces of $(\bR^{2n},\omega_0)$ where $\omega_0$ is the standard linear symplectic structure on $\bR^{2n}$. The symplectic group $Sp(2n,\bR)$ acts transitively on $LGr(\bR^{2n})$ on the left, and the stabilizer of $\bR^n\times \{0\}\subset \bR^{2n}$ is the standard Siegel parabolic subgroup
$$
P_n = \Big\{ \left[\begin{array}{cc} A & B\\ C & D\end{array}\right] \in  Sp(2n,\bR) \Big| \; C=0\Big\}
     = \Big\{ \left[\begin{array}{cc} A & B\\ 0 & (A^T)^{-1}\end{array}\right]\Big| \; A\in GL_n(\bR),  B\in M_n(\bR), AB^T=BA^T \Big\}.  
$$
Therefore, 
$$
LGr(\bR^{2n}) = Sp(2n,\bR)/P_n.
$$
The right hand side is a homogeneous space, which is equipped with the structure of a smooth manifold determined by the smooth structure on the Lie group $Sp(2n,\bR)$. We have
\begin{equation} \label{eqn:Sp-G}
\dim LGr(\bR^{2n}) = \dim Sp(2n,\bR)- \dim P_n = (2n^2+n) - \frac{3n^2+n}{2} =\frac{n(n+1)}{2}. 
\end{equation}

Let $LGr(\bC^n)$ denote the set of all linear Lagrangians in $\bC^n$ with the standard linear symplectic structure $\omega_0$ which is compatible with the standard complex structure on $\bC^n$. These structures determine the standard hermitian inner product on $\bC^n$. The unitary group $U(n)$ acts transitively on the left on $LGr(\bC^n)$, and the stabilizer of $\bR^n\subset \bC^n$ is the orthogonal group $O(n)$, so 
$$
LGr(\bC^n) = U(n)/O(n).
$$
We have
\begin{equation}\label{eqn:U-O}
\dim LGr(\bC^n) = \dim U(n)-  \dim O(n) = n^2 - \frac{n(n-1)}{2} =\frac{n(n+1)}{2}
\end{equation}
which is consistent with \eqref{eqn:Sp-G}. The diffeomorphism  $LGr(\bC^n) \cong LGr(\bR^{2n})$ is induced by the inclusion 
$$
U(n) \hookrightarrow Sp(2n, \bR), \quad A + i B \mapsto \left[\begin{array}{cc} A & -B \\ B & A\end{array} \right].
$$
as a closed Lie subgroup. 

\subsection{The Lagrangian Grassmannian bundle over a symplectic manifold}\label{subsec:LGr over symplectic mfd}
Let $(X,\omega)$ be a symplectic manifold of real dimension $2n$. Given any point $p \in X$, a symplectic basis of $T_p X$ is an ordered $\bR$-basis 
$(e_1,\ldots, e_{2n})$ of $T_p X\cong \bR^{2n}$ such that
for $i, j\in \{1,\ldots,n\}$
$$
\omega(p)(e_i, e_j) = \omega(p)(e_{n+i}, e_{n+j})=0, \quad 
\omega(p)(e_i, e_{n+j}) = \delta_{ij}. 
$$
We define the symplectic frame bundle $F(TX,\omega)\to X$ to be the fiber bundle
whose fiber over $p\in X$ is the set of all symplectic basis of $(T_p X, \omega(p))$. Then $F(TX,\omega)\to X$ is a principal $\Sp(2n,\bR)$-bundle over $X$.
The Lagrangian Grassmannian $LGr(X,\omega)\to X$ is the fiber bundle
whose fiber over $p$ is the set of all Lagrangian subspaces of $(T_pX, \omega(p))$. We have
$$
LGr(X,\omega) = F(TX,\omega)\times_{Sp(2n,\bR)} LGr(\bR^{2n}) = F(TX,\omega)\times_{Sp(2n,\bR)}(Sp(2n,\bR)/P_n) = F(TX,\omega)/P_n. 
$$
Note that this is a special case of Example \ref{ex:homogeneous-space}.

\subsection{The Lagrangian Grassmannian bundle over an almost K\"{a}hler manifold}\label{sec:LGr}
Let $(X, \omega, J)$ be an almost K\"{a}hler manifold of real dimension $2n$, where $\omega$ is a symplectic form  and $J$ is an $\omega$-compatible almost complex structure. 
Then $T_\bC X = (TX,J)$ is a Hermitian vector bundle of rank $n$ 
over $X$, and $\det(T_\bC X) =\Lambda^nT_\bC X$ is a hermitian line bundle over $X$.
Let $F(TX,\omega,J)\to X$ be the unitary frame bundle whose fiber
over $p\in X$ is the set of all ordered orthonormal $\bC$-basis of $T_\bC X$. Then $F(TX,\omega,J)\to X$ is a principal $U(n)$-bundle over $X$. We have 
$$
F(TX,\omega)= F(TX,\omega,J)\times_{U(n)}Sp(2n,\bR)
$$
which is a special case of Example \ref{ex:structure-group}, and 
$$
LGr(X,\omega) = F(TX,\omega, J)\times_{U(n)} LGr(\bC^n) = F(TX,\omega,J)\times_{U(n)} (U(n)/O(n))
= F(TX,\omega, J)/O(n)
$$
which is a special case of Example \ref{ex:homogeneous-space}. We also have
$$
F(TX,\omega,J)\times_{\det}\bC = \det(T_\bC X)
$$
which is a special case of Example \ref{ex:associated-vector-bundle}.
The determinant $\det: U(n)\to U(1)$ induces a surjective smooth map
$$
LGr(\bR^{2n}) = U(n)/O(n) \to U(1)/O(1) \cong U(1)
$$
where the isomorphism $U(1)/O(1)\cong U(1)$ is given by $aO(1)\mapsto a^2$. Here $a O(1) = \{\pm a\}$ is the left coset of $a\in U(1)$. This gives rise to 
\begin{equation}\label{eqn:phi}
\phi: U(n)/O(n) \lra U(1), \quad a O(n) \mapsto \det(a)^2
\end{equation}
which induces an isomorphism
$$
\phi_*: \pi_1(U(n)/O(n))\lra \pi_1(U(1))\cong \bZ. 
$$
The map \eqref{eqn:phi} is $U(n)$-equivariant: for any $a,b\in U(n)$,
$$
\phi(a\cdot bO(n)) =\phi(ab O(n)) = \det(ab)^2 = \det(a)^2 \det(b)^2 = \det(a)^2 \phi(bO(n)). 
$$
We obtain a map from $LGr(X,\omega)$, the Lagrangian Grassmannian bundle
of $(X,\omega)$, to the unitary frame bundle 
$F(\det(T_\bC X)^{\otimes 2})$ of the Hermitian line bundle $\det(T_\bC X)^{\otimes 2}$: 
$$
\phi_X: LGr(X,\omega) = F(TX,\omega,J)/O(n)
\lra F(\det(T_\bC X)^{\otimes 2})= F(TX,\omega,J)\times_{\det^2} U(1)
$$
The following conditions are equivalent:
\begin{enumerate}
\item $F(\det(T_\bC X)^{\otimes 2})$ is a trivial $U(1)$-bundle over $X$.
\item $\det(T_{\bC}X)^{\otimes 2} = F(TX,\omega, J)\times_{\det^2}\bC$ is a trivial complex line bundle over $X$.
\item $2c_1(T_\bC X) = 0 \in H^2(X;\bZ)$.
\end{enumerate}

In the following table we compare notation in this paper and in relevant sections in 
\cite{seidel_lagr, fuk_intro, math257b}.

\begin{center}

\begin{tabular}{|c|c| c| c| c| c| c|} \hline
This paper &  $LGr(\bR^{2n})$ & $X$ & $F(TX,\omega)$ & $LGr(X,\omega)$ & $F(TX, \omega, J)$ 
& $\det(T_\bC X)^{\otimes 2}$ \\ \hline
\cite[Section 2]{seidel_lagr} & $\cL(2n)$ & $M$  & $P$ & $\cL$ & $P_U$  & $\Delta$\\ \hline
\cite[Section 1.3]{fuk_intro} & $LGr(n)$ & $M$  & & $LGr(TM)$ &  &  \\ \hline
\cite[Section 2.6]{math257b} & $\Lambda(n)$ & $X$ & & $\cL$  & & $\Lambda_{\bC}^n(TX)^{\otimes 2}$ \\ \hline
\end{tabular}
\end{center}

\subsection{Existence of $\bZ$-grading.}\label{sec:Z-grading-details}
The Lagrangian Floer cohomology group does not admit a $\bZ$-grading in general. Following M. Kontsevich and P. Seidel \cite[Example 2.9]{seidel_lagr}, we now make the additional assumption that $2c_1(TX)=0$, which implies the existence of a global unitary frame $\Theta$ of the Hermitian line bundle 
$\det(T_\bC^*X)^{\otimes 2}$ (the dual of 
$\det(T_\bC X)^{\otimes 2}$). This defines a trivialization 
of $F(\det(T_\bC X)^{\otimes 2})$:
$$
\Theta: F(\det(T_\bC X)^{\otimes 2}) \stackrel{\cong}{\lra} X\times U(1). 
$$
Let 
$$
\exp(2\pi i\cdot) : \bR\lra U(1), \quad t \mapsto e^{2\pi i t}
$$
be the universal cover. We have the following commutative diagram
$$
\xymatrix{
\widetilde{LGr}(X,\omega) \ar[d]^{\pi_{LGr}} \ar[r]^{\widetilde{\phi}_X\quad} &
\widetilde{F}(\det(T_\bC X)^{\otimes 2}) \ar[d]\ar[r]^{\;\widetilde{\Theta}}_{\cong} & 
X\times \bR \ar[d]^{id_X\times \exp(2\pi i\cdot) } \ar[r]^{\widetilde{p}_2} & \bR\ar[d]^{\exp(2\pi i\cdot)} \\
LGr(X,\omega) \ar[r]^{\phi_X\quad} &
F(\det(T_\bC X)^{\otimes2}) \ar[r]^{ \; \Theta}_{\cong} & 
X\times U(1) \ar[r]^{p_2}  & U(1)
}
$$
In the above diagram, all three squares are Cartesian. Define\footnote{In \cite{seidel_lagr}, $\widetilde{LGr}(X,\omega)$ is called
an $\infty$-fold Maslov covering and denoted $\cL^\infty$. The map
$\alpha_\Theta$ is denoted $\det^2_\Theta$ in \cite{seidel_lagr}.}
$$
\alpha_\Theta:= p_2\circ \Theta\circ \phi_X: LGr(X,\omega)\to U(1). 
$$
Given any Lagrangian $L\subset X$, define $s_L: L\lra LGr(X,\omega)$ by $s_L(p) = (p, T_p L)$, and define
$$
\alpha_L:= \alpha_\Theta \circ s_L:  L\to U(1).
$$ 
Let $\nu\in H^1(U(1);\bZ)$ be the Poincar\'{e} dual of the class of a point $[\mathrm{pt}]\in H_0(U(1);\bZ)$.
Then $H^1(U(1);\bZ)=\bZ\nu$. The image of $\nu$ under the injective group homomorphism
$H^1(U(1);\bZ) \cong \bZ \lra H^1(U(1);\bR)\cong \bR$  is represented by the closed 1-form $dt$. The Maslov class of $L$ is defined by 
$$
\mu_L:= \alpha_L^*\nu \in H^1(L;\bZ).
$$
The following conditions are equivalent:
\begin{enumerate}
\item $\mu_L=0$.
\item $\alpha_L$ is homotopic to a constant map.
\item  There exists a smooth map $\widetilde{\alpha_L}: L\to \bR$ such that $e^{2\pi i \widetilde{\alpha}_L} =\alpha_L$. 
\item There exists a smooth map 
$\widetilde{s_L}: L \to \widetilde{LGr}(X,\omega)$ such that
$\pi_{LGr}\circ \widetilde{s_L} = s_L$.
\end{enumerate}

A graded Lagrangian in $(X,\omega, J, \Theta)$ is a pair $(L,\widetilde{\alpha_L})$, where $L$ is a Lagrangian in $(X,\omega)$ with
$\mu_L=0$, and $\widetilde{\alpha_L}: L\lra \bR$ is a real-valued function satisfying
the condition $e^{2\pi i\widetilde{\alpha_L} }=\alpha_L: L\to U(1)$. The function $\widetilde{\alpha_L}$ is unique up to addition of a constant integer. 

Let $\tL_0= (L_0, \widetilde{\alpha_{L_0}})$ and $\tL_1= (L_1,\widetilde{\alpha_{L_1}})$ be two graded Lagrangians in $(X,\omega, J, \Theta)$ such that 
$L_0$ and $L_1$ intersect transversely at finite many points, and let $p\in L_0\cap L_1$.  Then $T_p L_0$ and $T_p L_1$ are linear Lagrangian subspaces
of $T_pX$, and $T_p L_0 \cap T_p L_1 =\{0\} \subset T_p X$. There exists
a linear symplectomorphism $A: T_p X \to \bC^n$ such that 
$A(T_p L_0) =\bR^n$ and $A(T_p L_1) = i\bR^n$.  Following \cite[Section 1.3]{fuk_intro}, we define  
$$
\lambda: [0,1]\to LGr(T_p X), \quad  
t\mapsto  A^{-1}\left( e^{-i\pi t/2} \bR^n\right).
$$
Then $\lambda(0)=T_p L_0$ and $\lambda(1)=T_p L_1$. The path $\lambda$ is known as a {\em canonical short path} from $T_p L_0$ to $T_p L_1$. Any two canonical short paths are fixed point homotopic. Let $\alpha_\lambda := \alpha_\Theta \circ \lambda: [0,1]\to U(1)$.
By path lifting property of the covering map $\exp(2\pi i \cdot):\bR\to U(1)$, there exists a unique $C^\infty$ map $\widetilde{\alpha}_\lambda: [0,1]\to \bR$ such that  
$$
e^{2\pi i\widetilde{\alpha_\lambda} } = \alpha_\lambda, \quad
\widetilde{\alpha_\lambda}(0)= \widetilde{\alpha_{L_0}}(p).
$$
Then $\widetilde{\alpha_\lambda}(1) \in \bR$ is independent of choice of the canonical short path $\lambda$ since any two canonical short paths are fixed point homotopic. Define
$$
\deg(\tL_0, \tL_1; p):= \widetilde{\alpha_{L_1}}(p) - \widetilde{\alpha_\lambda}(1).
$$
Then $\deg(\tL_0, \tL_1; p)\in \bZ$ since
$$
e^{2\pi i\widetilde{\alpha_{L_1}}(p) } = e^{2\pi i \widetilde{\alpha_\lambda}(1) } = \alpha_{L_1}(p). 
$$

The following lemma is well-known. We include a proof for completeness. 
\begin{lemma}\label{A-model-Serre}
$\deg(\tL_0, \tL_1; p) + \deg(\tL_1, \tL_0; p)=n.$ 
\end{lemma}

\begin{proof} Let $\lambda_+: [0,1]\to LGr(T_p X)$ be a canonical short path
from $T_pL_0$ to $T_p L_1$, and let $\lambda_-: [0,1]\to LGr(T_pX)$ be a canonical short path
from $T_p L_1$ to $T_p L_0$. Let $\widetilde{\alpha_{\lambda_\pm}}: [0,1]\to \bR$  be the unique path such that 
\begin{equation}\label{eqn:lambda}
e^{2\pi i \widetilde{\alpha_{\lambda_\pm}}} = \alpha_{\lambda_\pm},\quad
\widetilde{\alpha_{\lambda_+}}(0) = \widetilde{\alpha_{L_0}}(p),\quad \widetilde{\alpha_{\lambda_-}}(0)= \widetilde{\alpha_{L_1}}(p).
\end{equation}
Then 
\begin{equation}\label{eqn:minus-n}
\widetilde{\alpha_{\lambda_+}}(1) -\widetilde{\alpha_{\lambda_+}}(0)
+ \widetilde{\alpha_{\lambda_-}}(1) -\widetilde{\alpha_{\lambda_-}}(0) =-n,
\end{equation}
and
\begin{equation}\label{eqn:deg+deg}
\deg(\tL_0, \tL_1; p) + \deg(\tL_1, \tL_0; p)
= \widetilde{\alpha_{L_1}}(p) - \widetilde{\alpha_{\lambda_+}}(1)
+ \widetilde{\alpha_{L_0}}(p) - \widetilde{\alpha_{\lambda_-}}(1)
\end{equation}
The lemma follows from \eqref{eqn:lambda}, \eqref{eqn:minus-n}, and \eqref{eqn:deg+deg}.
\end{proof}

\begin{remark}
Lemma \ref{A-model-Serre} implies Lemma \ref{Poincare} which is the A-model version of Serre duality. It is the reason why the canonical short path rotates clockwise instead of counterclockwise. 
\end{remark}

\begin{example}[linear Lagrangians in a flat symplectic torus]\label{ex:grading_torus}
For any integer $k\in \bZ$, let 
$$
\ell_k =\{ \theta = -kr\} \subset T^{2n} = \bR^{2n}/\bZ^{2n}
$$
where $\theta=(\theta_1,\ldots, \theta_n)$ and $r=(r_1,\ldots, r_n)$ are the coordinates
on $\bR^{2n}$.  Let 
$$
\omega =\sum_{j=1}^n dr_j\wedge d\theta_j. 
$$
Then $\ell_k$ is a Lagrangian submanifold of $(T^{2n}, \omega)$.

Let $\phi_k = \displaystyle{ \frac{1}{\pi}\arctan(k) \in (-\frac{1}{2}, \frac{1}{2}) }$. Then 
$$
\cos(\pi \phi_k) = \frac{1}{\sqrt{1+k^2}}, \quad \sin(\pi \phi_k) =\frac{k}{\sqrt{1+k^2}}.
$$
Note that if $k_0<k_1$ then $\phi_{k_0}< \phi_{k_1}$, and that  $\phi_{-k} = -\phi_k$. In particular, 
$\phi_0=0$.

The Lagrangian Grassmannian bundle $ LGr:= LGr(T^{2n},\omega)$ is trivial. 
$\alpha_\Theta: LGr\to U(1)$ is given by 
$$
\alpha_\Theta: LGr= T^{2n}\times U(n)/O(n) \lra U(1),
\quad (x, aO(n)) \mapsto \det(a)^2,
$$
and $s_{\ell_k}: \ell_k\to LGr$ is given by 
$$
s_{\ell_k}: \ell_k \to LGr = T^{2n}\times U(n)/O(n), \quad
p\to (p, e^{-\pi i\phi_k}I_n O(n)). 
$$
Therefore, 
$$
\alpha_{\ell_k} =\alpha_\Theta \circ s_{\ell_k}: \ell_k\to U(1)
$$
is the constant map to $e^{-2\pi i n\phi_k}$.  We equip $\ell_k$ with the grading
$\widetilde{\alpha_{\ell_k}}: \ell_k \to \bR$ which is the constant map to $-n\phi_k$.  

Let $k_0, k_1 \in \bZ$, $k_0\neq k_1$. The $\ell_{k_0}$ and $\ell_{k_1}$ intersect transversally.
For any $p\in \ell_{k_0}\cap \ell_{k_1}$, the canonical short path from $T_p\ell_{k_0}$ to $T_p\ell_{k_1}$ is given by 
$$
\lambda(t) =\begin {cases}
e^{-\pi i  (\phi_{k_1}-\phi_{k_0}) t} T_p \ell_{k_0} & \text{if } k_1>k_0,\\
e^{-\pi i (1+\phi_{k_1}- \phi_{k_0}) t} T_p \ell_{k_0} & \text{if }k_1<k_0.
\end{cases}
$$
So $\alpha_\lambda = \alpha_\Theta \circ \lambda: [0,1]\to U(1)$ is given by 
$$
\alpha_{\lambda}(t) = \begin{cases}
e^{-2\pi i n (\phi_{k_1}-\phi_{k_0}) t} e^{-2\pi i n\phi_{k_0}} & \text{if } k_1>k_0,\\
e^{-2\pi i n (1+\phi_{k_1}-\phi_{k_0}) t} e^{-2\pi i n \phi_{k_0}} & \text{if }k_1<k_0.
\end{cases}
$$
Let $\widetilde{\alpha_{\lambda}}: [0,1]\to \bR$ be the unique lift of $\alpha_{\lambda}$ with
$\widetilde{\alpha_{\lambda}}(0) =  \widetilde{ \alpha_{\ell_{k_0}}}(p) = -n\phi_{k_0}$. Then 
$$
\widetilde{\alpha_{\lambda}}(t) = \begin{cases}
-n  (\phi_{k_1}-\phi_{k_0}) t - n\phi_{k_0} = -n\left( t\phi_{k_1} + (1-t)\phi_{k_0}\right) & 
\text{if } k_1>k_0, \\
-n (1+\phi_{k_1}-\phi_{k_0}) t - n\phi_{k_0} =-n \left( t(1+\phi_{k_1}) + (1-t)\phi_{k_0} \right)  & 
\text{if } k_1<k_0.
\end{cases}
$$
$$
\deg\left(  ( \ell_{k_0},-n\phi_{k_0}), ( \ell_{k_1}, -n\phi_{k_1}); p \right) = \widetilde{\alpha_{\ell_{k_1}}}(p) -\widetilde{\alpha_{\lambda}}(1) = -n\phi_{k_1} -\widetilde{\alpha_{\lambda}}(1) =
\begin{cases}
0 &\text{if } k_1>k_0,\\
n & \text{if }k_1<k_0.
\end{cases}
$$

\end{example}

\subsection{Fukaya category with $\bZ$-grading}
In the remainder of this paper, we consider $\uX = (X,\omega_{\bC} = B + i\omega,  J, \Theta)$ 
where $\omega$ is a symplectic structure, $B$ is a closed real 2-form, $J$ is an $\omega$-compatible almost complex structure,
and  $\Theta$ is a unitary frame of the Hermitian line bundle $\det(T_{\bC}X)^{\otimes 2}$. An object in 
the Fukaya category $\Fuk(\uX)$ of $\uX$ is 
$$
\hat{L} = (L, \widetilde{\alpha_L}, \cL, \nabla)
$$
where $(L, \cL, \nabla)$ is an object in $\Fuk(X,\omega_\bC)$ and $\widetilde{\alpha_L}: L\to \bR$ is a grading, so that the pair $(L, \widetilde{\alpha_L})$ is a graded
Lagrangian in $(X, \omega, J, \Theta)$. Given two objects
$$
\hL_0=(L_0, \widetilde{\alpha_{L_0} }, \cL_0, \nabla_0),
\quad \hL_1 = (L_1, \widetilde{\alpha_{L_1}}, \cL_1, \nabla_1)
$$
of $\Fuk(\uX)$, define
$$
CF^i(\hL_0, \hL_1) =\bigoplus_{\substack{ p\in L_0\cap L_1\\
\deg(\tL_0, \tL_1;p)=i} } \Hom_{\bC}( (\cL_0)_p, (\cL_1)_p),
$$
where $\tL_j = (L_j, \widetilde{\alpha_{L_j}})$.
Define
$$
\Hom_{\Fuk(\uX)}(\hL_0, \hL_1) := CF^*(\hL_0, \hL_1) = \bigoplus_i CF^i(\hL_0, \hL_1)
$$
which is $\bZ$-graded vector space over $\bC$. By \eqref{eqn:duality-p} and Lemma \ref{A-model-Serre}, we have the following graded version of \eqref{eqn:duality-CF}.
\begin{lemma} \label{Poincare}
$CF^i(\hL_1, \hL_0) = CF^{n-i}(\hL_0, \hL_1)^\vee$.
\end{lemma}

\section{Products}\label{subsec:products}

Let
$$
\hat{L}_j = (L_j,  \widetilde{\alpha_{L_j}}, \cL_j,  \nabla_j), \quad 0\leq j\leq k
$$
be $k+1$ objects in $\Fuk(\uX)$ such that for $j=0,\ldots,k$,
$L_j$ and $L_{j+1}$ intersect transversally at finite many points, where $L_{k+1}=L_0$.
In the remainder of Section \ref{subsec:products}, we fix $(\hat{L}_0,\ldots, \hat{L}_k)$ and
$(p_0,\ldots, p_k)$, where $p_0\in L_k\cap L_0$ and $p_j \in L_{j-1}\cap L_j$ for $j=1,\ldots,k$. 

In this section we introduce the $A_\infty$ product $\mu^k:CF(\hat L_{k-1}, \hat L_k)\otimes \cdots \otimes CF(\hat L_0, \hat L_1)\to CF(\hat L_0, \hat L_k)
$, which has inputs $\rho_j\in \Hom_\bC((\cL_{j-1})_{p_j},(\cL_j)_{p_j})$ for $j=1,\ldots, k$.   The map $\mu^k$ is given by the count of $J$-holomorphic discs $u$ with boundaries on the Lagrangians $L_0,\ldots, L_k$ and corners at $p_0, \ldots, p_k$ ($p_0$ is a generator of the output), weighted by $\rho(\rho_k,\ldots, \rho_1; u)\in \Hom_\bC((\cL_{0})_{p_0},(\cL_k)_{p_0})$. 
  The weight $\rho(\rho_k,\ldots, \rho_1; u)$ depends on the inputs $\rho_j$'s, the complexified symplectic area of $u$ with respect to $\omega_\bC$, and the holonomy of the connection along $\partial u$  which depends on $B$; more explicitly see Definition \ref{def: output}.  In Lemma \ref{rho}, we show that this weight only depends on the homotopy class of $u$ relative to the boundary conditions.  Thus, $\mu^k$ as precisely stated in Definition \ref{def:structure maps} is a well defined count over the homotopy classes $[u]$.

Recall that a triple of topological spaces $(X,A,B)$ consists of a topological space $X$ and two subspaces
$A, B$ with $B\subset A\subset X$. A continuous map of triples $f:(X,A,B)\to (Y, G, H)$ is a continuous map $f:X\to Y$ such 
that $f(A)\subset G$ and $f(B)\subset H$.  

\begin{definition}\label{smooth-disks} 
Let $\cD(X, (L_0,\ldots, L_k), (p_0,\ldots,p_k))$ be the set of continuous maps of triples
$$
u: \Big(\bD, \partial \bD =\bigcup_{j=0}^k \partial_j\bD, \{ z_0,\ldots, z_k\} \Big) \lra \Big(X, \bigcup_{j=0}^k L_j , \{ p_0,\ldots,p_k\} \Big)  
$$
satisfying the following properties. 
\begin{enumerate} 
\item[D1.] $\bD$ is the unit disk in $\bC$, oriented by the complex structure. This determines
an orientation on its boundary $\partial \bD \simeq S^1$. 
\item[D2.] $z_0, z_1,\ldots, z_k$ are distinct points on  the boundary $\partial \bD$ which respect the cyclic order determined by the orientation on $\partial \bD$.
\item[D3.]  For $j=0,\ldots, k$, $\partial_j \bD \subset  \partial \bD$ is the arc  between $z_j$ and $z_{j+1}$, where $z_{k+1}=z_0$. 
 \item[D4.] $u: \bD \to X$ is a piecewise smooth continuous map. For $j=0,\ldots,k$, 
$u(\partial_j \bD) \subset L_j$ and $u(z_j) = p_j$.
\item[D5.] $\partial_j u := u|_{\partial_j \bD} : \partial_j \bD \simeq [0,1] \to L_j$ is a piecewise smooth path from $p_j$ to $p_{j+1}$, where $p_{k+1}=p_0$. 
\end{enumerate} 
\end{definition}

\begin{definition} \label{smooth-homotopy} 
A {\em homotopy} between two maps $u, u'$ in $\cD(X, (L_0,\ldots, L_k), (p_0,\ldots, p_k))$ is a continuous map of triples
$$
h: \Big(\bD, \partial \bD =\bigcup_{j=0}^k \partial_j\bD , \{ z_0,\ldots, z_k\} \Big) \times[0,1]  \lra \Big(X, \bigcup_{j=0}^k L_j , \{ p_0,\ldots,p_k\} \Big)  
$$
satisfying the following conditions. 
\begin{enumerate}
\item[H1.] For any $t\in [0,1]$, define  $h_t:\bD\lra X$ by $h_t(z)= h(z,t)$.  Then
$h_t$ is a map in $\cD(X,(L_0,\ldots,L_k), (z_0,\ldots,z_k))$, $h_0=u$, and $h_1=u'$. 
\item[H2.] $h$ is a piecewise smooth continuous map.
\item[H3.]  The map $\partial_j h= h|_{\partial_j \bD \times [0,1]}: \partial_j \bD \times [0,1] \simeq [0,1]\times [0,1] \to L_j$ is piecewise smooth.   
\end{enumerate} 
We say two maps in $\cD(X, (L_0,\ldots, L_k), (z_0,\ldots, z_k))$ are  {\em homotopic} if there exists a homotopy between them.
\end{definition}

Note that in H3, $\partial_j h$ is a fixed-end-point homotopy between the two paths $\partial_j u$ and $\partial_j u'$ in $L_j$
 from $p_j$ to $p_{j+1}$.

\begin{definition}[$J$-holomorphic]
Let $J$ be an $\omega$-compatible almost complex structure on $X$. Let 
$$
u: \Big(\bD, \partial \bD =\bigcup_{j=0}^k \partial_j\bD, \{ z_0,\ldots, z_k\} \Big) \lra \Big(X, \bigcup_{j=0}^k L_j , \{ p_0,\ldots,p_k\} \Big)  
$$
be a map in $\cD(X, (L_0,\ldots, L_k), (p_0,\ldots,p_k))$. We say $u$ is $J$-holomorphic if
$u$ is smooth on $\bD-\{z_0,\ldots, z_k\}$, and 
\begin{equation}\label{eqn:J-hol}
J\circ du =du \circ \bj, 
\end{equation}
where $\bj$ is the standard complex structure on $\bD$.
\end{definition}
\begin{remark}
Define
\begin{equation} \label{eqn:dbar-J}
\bar{\partial}_J(u) := \frac{1}{2}(du + J\circ du \circ \bj).
\end{equation}
Then Equation \eqref{eqn:J-hol} is equivalent to 
$$
\bar{\partial}_J(u)=0.
$$
\end{remark}

Let   
\begin{equation}
\cM_J:= \cM(\uX, (\hL_0,\ldots, \hL_k), (p_0,\ldots, p_k); [u], J)
\end{equation}
be moduli of $J$-holomorphic maps in  $\cD(X, (L_0,\ldots, L_k), (p_0,\ldots, p_k)$  in the homotopy class $[u]$. 

For $k\geq 2$, there is a forgetful map
\begin{equation}\label{eqn:forget}
\mathrm{Forget}: \cM(\uX, (\hL_0, \ldots, \hL_k), (p_0, \ldots, p_k); [u], J) \lra \fT_{k+1}
\end{equation}
sending a $J$-holomorphic map to its domain, where 
$$
\fT_{k+1} = \left.\left\{ 
(z_0,z_1,\ldots, z_k) \in (\partial \bD)^{k+1} \Big|
\begin{array}{l}
z_i\neq z_j \text{ if }i\neq j, \quad z_0, z_1,\ldots, z_k \text{ respect} \\
\text{the  counterclockwise cyclic order of }\partial\bD
\end{array} \right\} \right/\sim.
$$
$\fT_{k+1}$ is a smooth manifold of dimension $k-2$ and is homeomorphic to $\bR^{k-2}$
\cite[Lemma 1.3]{fuk_oh}.  

Let $\tL_j = (L_j, \widetilde{\alpha_{L_j}})$.
The expected dimension of the fiber of the forgetful map \eqref{eqn:forget} is 
\begin{equation} \label{eqn:index-u}
\mathrm{ind}([u]) = \deg(\tL_0, \tL_k,p_0) -\sum_{j=1}^k \deg(\tL_{j-1}, \tL_j; p_j). 
\end{equation}
(See e.g. \cite[Section 2.2]{fuk_intro}.) 
Therefore, the expected dimension of $\cM_J$ is
$$
\mathrm{ind}([u]) + \dim_{\bR}(\fT_{k+1}) =  \deg( \tL_0, \tL_k, p_0) - \sum_{j=1}^k \deg( \tL_{j-1}, \tL_j ; p_j)  + k-2.
$$
Note that the right hand side does not depend on $[u]$. 

Given positive integers $l, p$ where $lp>2$, let  
$$
\cW^{l,p} :=\cW^{l,p}(X, (L_0,\ldots, L_k), (p_0,\ldots, p_k); [u])
$$
be the moduli of continuous maps of triples
$$
u: \Big(\bD, \partial \bD =\bigcup_{j=0}^k \partial_j\bD, \{ z_0,\ldots, z_k\} \Big) \lra \Big(X, \bigcup_{j=0}^k L_j , \{ p_0,\ldots,p_k\} \Big)  
$$
in a the homotopy class $[u]$ satisfying D1, D2, D3 in Definition \ref{smooth-disks} and
\begin{enumerate}
\item[D4'] $u: \bD \to X$ is of class $W^{l,p}$ for $j=0,\ldots,k$, 
$u(\partial_j \bD) \subset L_j$ and $u(z_j) = p_j$.
\item[D5'] $\partial_j u := u|_{\partial_j \bD} : \partial_j \bD \simeq [0,1] \to L_j$ is a path from $p_j$ to $p_{j+1}$, where $p_{k+1}=p_0$. 
\end{enumerate}
For any $\omega$-compatible almost complex structure $J$, $\overline{\dd}_J$ is a section of a Banach bundle
$\cE^{l-1,p}$ where 
$$
\cE^{l-1,p}_{z_0,\ldots, z_k, u}=W^{l-1,p}(\bD, u^*TX\otimes \Lambda^{0,1}_\bD).
$$
and $\cM_J = \overline{\dd}_J^{-1}(0)\subset \cW^{l,p}$. If the linearization of $\overline{\dd}_J$ is surjective for all $(z_0,\ldots,z_k,u)\in  \mathcal{M}_J$, then $J$ is called \underline{regular}.
\begin{theorem} 
If $J$ is regular, then $\mathcal{M}_J$ is a smooth manifold of dimension 
$$
\deg( \tL_0, \tL_k, p_0) - \sum_{j=1}^k \deg( \tL_{j-1}, \tL_j ; p_j)  + k-2
$$
equipped with an orientation determined by the orientations and spin structures of
$L_0,\ldots, L_k$. 
\end{theorem}
In order to use this theorem, one needs existence of regular $J$ by showing they are dense in the set of all $J$ in the setting considered, hence the use of the word ``generic" for regular $J$.  The proof depends on the symplectic manifold and moduli space in question; it can be proven by examining the possible configurations of $u$ and solving a Fredholm problem with larger spaces in which both $(z_0,\ldots,z_k, u)$ and $J$ vary. Examples include \cite{fuk_oh}, \cite[ Theorem 3.1.6, Appendix C.4 p 604]{mcduff2012j} (for pseudo-holomorphic spheres), \cite[Equation (2.5)]{fuk_intro}, \cite{FloerHoferSalamon} (for Hamiltonian Floer theory), and \cite[p63]{Ca} (for discs and discs attached to spheres).

\color{black}

\subsection{Weights in structure maps}\label{sec: rel_Stokes} 

A pseudo-holomorphic disc $u$ with a single Lagrangian boundary is counted in a Fukaya category structure map with weight $\rho(u)$,  which is precisely a pairing in relative cohomology:
$$
\rho(u) \equiv \exp(2\pi i \langle (B+i\omega,\theta_\nabla),u \rangle)
$$
where $\theta_\nabla$ is the connection 1-form associated to the connection $\nabla$. In particular, $\rho(u)$ is well-defined on relative homology, and we can track how it changes when moving a Lagrangian in Theorem \ref{prop:polygon-area} as well as how the structure map changes in Theorem \ref{thm:str_map_isotopy}. Note that rather than viewing the B-field $B$ as an element of $H^2(X;\bR)/H^2(X;\bZ)$ as one does for closed string invariants, for open string invariants of discs with Lagrangian boundary conditions, the B-field should be thought of as a pair $(B,\theta)\in H^2(X,L;\bR)$, where $\theta$ is a 1-form on the Lagrangian $L$. Then we require the pair $(B,\theta)$ to be closed as a relative de Rham form, which leads to the condition $B|_L = d\theta$.

Let $L\subset X$ be a Lagrangian of a symplectic manifold $(X,\omega)$. Denote $u: (\bD, \dd \bD) \to (X, L)$ a $J$-holomorphic disc for an almost complex structure $J$, and $\beta = [u(\bD, \dd \bD)]\in H_2(X,L; \bZ)$ a fixed relative homology class. The energy of the disc $E(u) = \int_\bD |du|_{J,\omega}^2$ equals its area $\int_\bD u^* \omega$ for $J$-holomorphic discs. We can replace $\omega$ with $\omega + iB$ where $B$ is any closed 2-form; it is therefore not necessarily non-degenerate nor vanishing on $L$. A connection 1-form is part of the data of a relative cohomology class, so that when $B=0$ this is a local system or flat connection 1-form. Associated to the disc we have a weight given by a pairing of \textit{relative} cohomology. 

The disc weight is independent of homology class by Stokes' theorem $\int_{\dd P} u^*\omega=0$ for a 3-cycle $P$ because $d \omega=0$. In other words, $\int_{u(S)} \omega = \int_{u'(S)} \omega$ when $[u] = [u'] \in H_2(X)$. This still holds if we only have $[u] = [u'] \in H_2(X,L)$, provided $L$ is Lagrangian: $[u]=[u']\in H_2(X,L)$ implies $\int_{u(S)}\omega-\int_{u'(S)}\omega =0$ because $\omega|_L=0$. More generally, given a $k$-form $\omega \in \Omega^k(X)$, and a $C^\infty$ map $f: (S, \partial S) \to (X, L)$ representing a relative homology class in $H_k(X,L;\bZ)$, the integral $\int_S  f^*\omega$ depends only on the class $[f]\in H_k(X,L;\bZ)$ if and only if $\omega$ is closed and $L$ is Lagrangian.

\begin{lemma}[Stokes' theorem for relative homology]\label{lem:Stokes_Lagr} Let $j: L \to X$ be inclusion.
 $\int_S  f^*\omega$ depends only on the class $[f]\in H_k(X,L;\bZ)$ if and only if $d\omega =0 \textup{ and } j^*\omega =0$.
\end{lemma}

\begin{proof}
Suppose that $f_0: (S_0, \pS_0)\to (X,L)$ and $f_1:(S_1,\pS_1)\to (X,L)$ are homologous relative $k$-cycles. By definition, this means there exists a piecewise $C^\infty$ map   $F: C \to X$ from a $(k+1)$-dimensional manifold with corners $C$, and there exists $S_2$ a $k$-dimensional manifold, such that 
$$
\partial C = S_1 -  S_0 + S_2,\quad F|_{S_0} = f_0, \quad F|_{S_1} = f_1, \quad F(S_2)\subset L,
$$
so that $f_2:=F|_{S_2}: S_2 \to L$ is a $k$-chain in $L$.  We have
$$
\int_C F^*d\omega = \int_C d (F^*\omega)  = \int_{S_1} f_1^*\omega - \int_{S_0} f_0^*\omega  + \int_{S_2} f_2^* j^*\omega,
$$
where we keep track of the fact that the image of $F|_{S_2}$ is in $L$, in the last term, to illustrate where the Lagrangian condition on $L$ will be relevant. Therefore,
\begin{equation}\label{eqn:C} 
\int_{S_1} f_1^*\omega - \int_{S_0} f_0^*\omega  =\int_C F^*d\omega   - \int_{S_2} f_2^* j^*\omega.
\end{equation}
A sufficient condition for this to be zero is that 
\begin{equation}\label{eqn:omega}
d\omega =0 \textup{ and } j^*\omega =0
\end{equation} 
which holds when $(X,\omega)$ is a symplectic manifold and $L$ is a Lagrangian submanifold.  Indeed,  the condition  \eqref{eqn:omega} is also necessary, if we want the right hand side of  \eqref{eqn:C}
to vanish for all $F$, such as for the following:
\begin{itemize}
\item  For any small $(k+1)$-ball $D$ disjoint from $L$ we must have $\int_D d\omega =0$. Therefore $d\omega=0$. 
\item For any small $k$-ball $D$ in $L$ we must have $\int_D j^*\omega=0$.  Therefore $j^*\omega=0$. 
\end{itemize}
\end{proof}

Now we extend the theory to relative cohomology. Recall as mentioned briefly in \cite[p 78]{BT82} that the relative de Rham complex of the pair $(X,L)$ is defined to be the following chain complex and differential
$$
\Omega^k(X,L) =\Omega^k(X) \oplus \Omega^{k-1}(L), \quad d(B,\theta) = (dB, j^*B-d\theta).
$$
\begin{example}
In particular, if $L$ is a Lagrangian submanifold of a symplectic manifold $(X,\omega)$ then
the pair $(\omega,0) \in \Omega^2(X)\oplus \Omega^1(L)$ is a 2-cocycle and represents a relative de Rham cohomology class in $H^2_{dR}(X,L)$. For the B-field, we are interested in the case of $\omega+iB \in \Omega^2(X)$, in which case the second component in $\Omega^1(L)$ will not be 0, as follows. 
\end{example} 

\begin{definition}[Relative cocycle, pairing in relative cohomology]\label{def:rel_pairing}
Suppose that  $(B, \theta) \in \Omega^k(X) \oplus \Omega^{k-1}(L)$ is a \textit{cocycle in the relative de Rham complex} $(\Omega^*(X,L),d)$, or equivalently,
$$
d B =0, \quad  j^*B = d\theta,
$$
and let $f:(S,\pS) \to (X,L)$ be a relative $k$-cycle. We define a \textit{pairing}
$$
\langle (B,\theta), f\rangle = \int_S f^*B -\int_{\pS} (\pf)^* \theta
$$
where $\pf := f|_{\partial S}:\partial S\to L$.  
\end{definition}

\begin{example} One example is when $\theta$ is the connection 1-form on a trivial line bundle on a Lagrangian $L$ with curvature $j^*B$. Note that $\theta$ is only defined on $L$ so we cannot write $d \theta = B$ away from $L$; hence we cannot apply Stokes' theorem to conclude that
$\int_S f^*B -  \int_{\pS}(\pf)^*\theta=0$.
\end{example}

\begin{lemma}[Stokes' theorem for relative cohomology]\label{lem:Stokes_Bfield} The pairing
$\langle (B,\theta), f\rangle$
depends only on the relative de Rham cohomology class $(B,\theta) \in H^k_{dR}(X,L)$ and the relative cohomology class $[f]\in H_k(X,L;\bZ)$.
\end{lemma}
\begin{proof} We first show that it is independent of choice of $f$, using the set-up from the proof of Lemma \ref{lem:Stokes_Lagr}:
\begin{eqnarray*}
0 &=& \int_C F^*dB = \int_C d(F^*B) \\
&\substack{F(S_2) \subset L \\=}&  \int_{S_1} f_1^* B -\int_{S_0} f_0^* B + \int_{S_2} f_2^* d\theta \\
&\substack{\text{Stokes}\\ =}&  \int_{S_1} f_1^* B -\int_{S_0} f_0^*B  +\int_{\pS_2} (\pf_2)^*\theta \\
&\substack{\eqref{eq:homology_vary} \\ =}&   \int_{S_1}f_1^* B  - \int_{S_0} f_0^* B 
+ \int_{\pS_0} (\pf_0)^*\theta -\int_{\pS_1}(\pf_1)^*\theta \\
&=&  \langle (B,\theta), f_1\rangle -\langle (B,\theta),f_0\rangle.
\end{eqnarray*} 
where we used
\begin{equation}\label{eq:homology_vary}
\pS_1 -\pS_0 +\pS_2 = \partial \partial C=0. 
\end{equation}
We next show it is independent of choice of $(B,\theta)$: suppose that
$$
(B',\theta') = (B,\theta) + d(A,\psi) = (B+dA, \theta + j^*A-d\psi). 
$$
Then 
\begin{eqnarray*}
\langle (B',\theta'),f\rangle &=& \int_S f^*(B+dA) -\int_{\pS}(\pf)^* (\theta + j^*A -d\psi) \\
&\substack{\text{Stokes'}\\=}& \int_S f^* B + \int_{S} d(f^*A) -\int_{\pS}(\pf)^*\theta   -\int_{\pS} (\pf)^*j^*A  + \int_{\partial \pS} (\pf)^* \psi 
\end{eqnarray*}
where
$$
\int_S d(f^*A) = \int_{\pS} (\pf)^*j^*A, \quad \partial\pS =0
$$ 
so 
$$
\langle (B',\theta'), f \rangle = \langle (B,\theta), f \rangle. 
$$
\end{proof} 

Now we define the weight for $u$ with multiple Lagrangian boundaries.

\color{black}

\begin{definition}\label{def: output}
Given 
$$ 
\rho_j \in \Hom_{\bC}( (\cL_{j-1})_{p_j}, (\cL_j)_{p_j}), \quad  j=1,\ldots,k,
$$
and a map $u$ in $\cD(X,(L_0,\ldots, L_k), (p_0,\ldots, p_k))$, 
define $\rho(\rho_k,\ldots,\rho_1;u) \in  \Hom_{\bC}( (\cL_0)_{p_0}, (\cL_k)_{p_0} )$ by 
\begin{equation} \label{def:rho_for_curve}
\rho(\rho_k,\ldots,\rho_1;u) = 
e^{ 2\pi i  \int_{D^2}u^*\omega_\bC} \Hol_{\nabla_k} (\partial_k u)
\circ \rho_k \circ \Hol_{\nabla_{k-1}}(\partial_{k-1} u) 
\circ \rho_{k-1} \circ \cdots \circ \rho_1 \circ \Hol_{\nabla_0}(\partial_0 u) 
\end{equation} 
where $\Hol_{\nabla_j}(\partial_j u)$ is the parallel transport along the path $\partial_j u: \partial_j \bD \to L_j$ 
from $p_j$ to $p_{j+1}$ defined by the unitary connection $\nabla_j$ on the complex line bundle $\cL_j$ on $L_j$.
\end{definition}

\begin{lemma} \label{rho} Let $\rho_j \in \Hom_{\bC}( (\cL_{j-1})_{p_j}, (\cL_j)_{p_j})$, $j=1,\dots, k$. If 
$u, u' \in \cD(X, (L_0,\ldots, L_k), (p_0,p_1,\ldots,p_k) )$ are  homotopic  then  
$$
\rho(\rho_k,\ldots,\rho_1;u) = \rho(\rho_k,\ldots,\rho_1;u')
$$
\end{lemma}
\begin{proof} Let $h:\bD\times[0,1]\to X$ be a homotopy between $u$ and $u'$. 
$L_j$ is a Lagrangian submanifold of $(X,\omega)$,  so $\omega|_{L_j}=0$, which implies 
\begin{equation}
\int_{\partial_j \bD \times [0,1]}  (\partial_j h)^* (\omega|_{L_j}) =0.
\end{equation} 
The curvature $F_{\nabla_j}$  of  $\nabla_j$ satisfies $F_{\nabla_j} = -2\pi i B|_{L_j}$, so  
\begin{equation} \label{eqn:Hol-Hol} 
\Hol_{\nabla_j}(\partial_j u')\circ \Hol_{\nabla_j}(\partial_j u) ^{-1} =  \exp\left(2\pi i\int_{ \partial_j \bD \times [0,1]} (\partial_j h)^*(B|_{L_j} )\right)  \mathrm{Id}_{(\cL_j)_{p_j} } 
\end{equation} 

where $\mathrm{Id}_{(\cL_j)_{p_j} }: (\cL_j)_{p_j} \to  (\cL_j)_{p_j}$ is the identity map.  

Note that $d h^*\omega = h^*d\omega =0$ and $d h^*B = h^*dB=0$. By Stokes' theorem,
\begin{equation} \label{eqn:int-omega}
0  =  \int_{\bD\times [0,1]} d (h^*\omega) = \int_{\bD} h_1^*\omega    - \int_{\bD} h_0^*\omega + \sum_{j=0}^k \int_{\partial_j \bD\times [0,1]} (\partial_j h)^*(\omega|_{L_j}) 
= \int_{\bD} (u')^*\omega  -\int_{\bD} u^*\omega 
\end{equation}
and
\begin{equation} \label{eqn:int-B} 
\begin{aligned}
0  =&  \int_{\bD \times [0,1]} d (h^*B) = \int_{\bD} h_1^*B    - \int_{\bD} h_0^*B + \sum_{j=0}^k \int_{\partial_i \bD\times [0,1]} (\partial_j h)^*(B|_{L_j}) \\
=& \int_{\bD} (u')^*B  -\int_{\bD} u^*B + \sum_{j=0}^k \int_{\partial_j \bD \times [0,1]}(\partial_j h)^*(B|_{L_j}). 
\end{aligned} 
\end{equation} 
Then
\begin{eqnarray*} 
&& \rho(\rho_k,\ldots,\rho_1;u') \\
&=&  e^{ 2\pi i  \int_{\bD}{u'}^*\omega_\bC}\Hol_{\nabla_k} (\partial_k u')
\circ \rho_k \circ \Hol_{\nabla_{k-1}}(\partial_{k-1} u') 
\circ \rho_{k-1} \circ \cdots \circ \rho_1 \circ \Hol_{\nabla_0}(\partial_0 u') \\
&=  & e^{ 2\pi i  \int_{\bD}u^*\omega_{\bC}} \prod_{j=0}^k \exp\left(-2\pi i\int_{\partial_j \bD \times [0,1]} (\partial_j h)^*(B|_{L_j}) \right) \\ 
&& \cdot \Hol_{\nabla_k} (\partial_k u') \circ \rho_k \circ \Hol_{\nabla_{k-1}}(\partial_{k-1} u')  \circ \rho_{k-1} \circ \cdots \circ \rho_1 \circ \Hol_{\nabla_0}(\partial_0 u') \\
&=&  e^{ 2\pi i  \int_{\bD}u^*\omega_{\bC}} \Hol_{\nabla_k} (\partial_k u) \circ \rho_k \circ \Hol_{\nabla_{k-1}}(\partial_{k-1} u)  \circ \rho_{k-1} \circ \cdots \circ \rho_1 \circ \Hol_{\nabla_0}(\partial_0 u)\\
&=& \rho(\rho_k,\ldots,\rho_1;u) 
\end{eqnarray*} 
where the second equality follows from  \eqref{eqn:int-omega} and \eqref{eqn:int-B}, while the third equality follows from \eqref{eqn:Hol-Hol}. 
\end{proof}

By Lemma \ref{rho}, $\rho(\rho_k,\ldots, \rho_1;u)$ depends only on the homotopy class $[u]$ of $u$. 
Therefore, 
$$
\rho(\rho_k,\ldots, \rho_1;[u]):= \rho(\rho_k,\ldots, \rho_1; u)
$$
is well-defined.   

\begin{definition}\label{def:structure maps}
Define
$$
\mu^k: CF^{i_k}(\hL_{k-1}, \hL_k)\otimes \cdots \otimes CF^{i_1}(\hL_0, \hL_1)\lra CF^{2-k+ \sum_{j=1}^k i_j}(\hL_0, \hL_k)
$$
by 
\begin{equation} \label{eqn:mu-k}
\mu^k (\rho_k,\ldots, \rho_1) 
=\sum_{\substack{ p_0\in L_0\cap L_k \\ \deg(\tL_0, \tL_k; p_0) 
=2-k+\sum_{j=1}^k i_j \\ [u] } }
\sharp \cM(\uX, (\hL_0,\ldots, \hL_k), (p_0,\ldots, p_k);[u])
\rho(\rho_k,\ldots, \rho_1;[u]).
\end{equation}
where $\sharp \cM(\cdots)$ is the number of points in the 
zero dimensional  moduli space $\cM(\cdots)$.

\begin{remark}\label{rem:convergence}
The right hand side of \eqref{eqn:mu-k} is an infinite sum of complex numbers.  We assume that there is an open subset $U$ in the space of symplectic forms on $X$ such that the right hand side of \eqref{eqn:mu-k} converges absolutely when $\omega \in U$, so that we may work over $\bC$. \textcolor{black}{Some non-exact examples include when $L_i$ are affine linear Lagrangians in abelian varieties \cite{ACLLc} or when they are moment map fibers of compact semi-Fano toric manifolds or toric Calabi-Yau manifolds. When $\omega$ is exact or monotone, structure maps can also be defined over $\mathbb{C}$. In general, they are only defined over a formal power series $\mathbb{C}[[T]]$.}
\end{remark} 

In particular,  
$$
\mu^1: CF^i(\hL_0, \hL_1)\lra CF^{i+1}(\hL_0, \hL_1)
$$
is the Floer differential, and 
$$
\mu^2: CF^{i}(\hL_1,\hL_2)\times CF^j(\hL_0, \hL_1)
\to CF^{i+j}(\hL_0, \hL_2)
$$
is the Floer product. 
\end{definition}

\color{black}
\begin{remark}
    For U-shaped Lagrangians $L$ in a symplectic Landau-Ginzburg model \cite{ACLLe} with aspherical fibers, a single Lagrangian $L$ cannot bound a disc so $\pi_2(X,L)=0$. In general, we may have $\pi_2(X,L) \neq 0$. In that case, $(\mu^1)^2 \neq 0$ due to additional disc bubbling, and the $A_\infty$-category is \emph{curved}. One can then perturb the structure maps $\mu^k$ with elements $b$ in the degree 1 of a chain-level model of $CF^*(L,L)$. This amounts to adding insertions, additional boundary constraints, to the discs counted in $\mu^k$. Then the condition that the perturbed $\mu^k$ form an $A_\infty$-category amounts to $b$ being a \emph{Maurer-Cartan} element. This $b$ defines a \emph{bounding cochain} and the deformed differential now squares to 0. See \cite[\textsection 3]{fooo1} for details. This process can be generalized to more than one Lagrangian. In particular, adding a B-field can be interpreted as a \emph{bulk deformation} of the structure maps for all Lagragangian branes by the addition of $e^{2\pi i \int_{D^2}u^*B}$ in Equation \eqref{def:rho_for_curve}. 
\end{remark}

\color{black}

\section{Isotopy between objects}\label{sec: isotopy}

We first recall the definition of an isotopy on page 111 in \cite{Hirsch}.

\begin{definition}
Let $L$, $X$ be smooth manifolds. An {\em isotopy} of $L$ in $X$ is a homotopy
$$
\psi : L\times [0,1]\to X, \quad \psi(p,t) =\psi_t(p)
$$
such that the related map 
$$
\tilde{\psi}  : L\times [0,1]\to X\times [0,1], \quad (p,t)\mapsto (\psi_t(p), t)
$$
is a smooth embedding. In particular, for any $t\in [0,1]$ the map $\psi_t: L\to X$ is a smooth embedding and $\psi_t(L)$ is a smooth submanifold of $X$. 
We call $\tilde{\psi}$ the  {\em track} of $\psi$ and say $\psi$ is an isotopy from $\psi^0$ to $\psi^1$.
\end{definition} 

\begin{definition} [Lagrangian isotopy]\label{def:Lag_isotopy}
If $(X,\omega)$ is a symplectic manifold of dimension $2n$ and $L$ is a smooth manifold of dimension $n$, 
a {\em Lagrangian isotopy}  in $X$ is an isotopy $\psi: L\times [0,1]\to X$ such that for all $t\in [0,1]$, $\psi_t(L)$ is a Lagrangian submanifold.  
\end{definition}

If $\psi: L\times [0,1]\to X$ is a Lagrangian isotopy then
$$
\psi^*\omega = b \wedge dt, 
$$
where $b$ is a  real closed 1-form on $L\times [0,1]$.  In terms of the local coordinates $x=(x_1,\ldots, x_n)$ on $L$ and the global coordinate $t$ on $[0,1]$, $b$ is of the form $b = \sum_{j=1}^n a_j(x,t) dx_j$.  For each fixed $t$, $b_t$ is a closed 1-form on $L$.  Following P. Seidel \cite{Seidel_isotopy}, we say the Lagrangian isotopy $\psi$ is {\em exact} if $b\wedge dt = d(Hdt)$ for some smooth  function $H: L\times [0,1]\to \bR$.  Denote by $H_t$, where $H_t(p)=H(p,t)$, a time-dependent smooth function on $L$.   When $\psi$ is exact, $b_t=dH_t$.

Let $B$ be a B-field on $(X,\omega)$, so that $\omega_\bC = B+ i\omega \in \Omega^2(X,\bC)$ is a complexified symplectic form. 
\begin{definition} \label{def:isotopy on object}
Let $\hat{L} = (L, \widetilde{\alpha_L}, \cL,  \nabla)$ and $\hat{L}'= (L', \widetilde{\alpha_{L'}}, \cL',  \nabla')$ be two objects in $\Fuk(\uX)$. We say
$\hat{L}$ is isotopic to $\hat{L}'$ if there exists a Lagrangian isotopy $\psi:  L\times [0,1]\to X$, 
a $C^\infty$ map $\tilde{\alpha}: L\times [0,1]\to \bR$, 
a $C^\infty$ complex line bundle $\tilde{\cL}$ on $L\times [0,1]$,
and a unitary connection $\tilde{\nabla}$ on the complex line bundle $\tilde{\cL}$ such that 
\begin{enumerate}
\item $\psi_0(L)=L$, $\psi_1(L)=L'$, and $\psi_0: L\to L$ is the identity map.
\item The orientation and spin structure on $L$ is the pullback of those
on $L'$ under the diffeomorphism $\psi_1: L\to L'$.
\item For all $(p,t)\in L\times [0,1]$,
$e^{2\pi i\tilde{\alpha}(p,t)} = \alpha_{\psi_t(L)}\circ \psi_t(p)$. For all $p\in L$, $\alpha(p,0)= \tilde{\alpha}_L(p)$ and $\tilde{\alpha}(p,1)= \tilde{\alpha}_{L'}\circ \psi_1(p)$.
\item The curvature $F_{\tilde{\nabla}}$ of $\tilde{\nabla}$ satisfies $F_{\tilde{\nabla}} = -2\pi i \psi^*B$. 
\item Given $t\in [0,1]$, define $\iota_t: L\to L\times [0,1]$ by $\iota_t(p) = (p,t)$. Then
$$
\iota_0^*(\tilde{\cL}, \tilde{\nabla}) = (\cL,  \nabla),\quad
\iota_1^*(\tilde{\cL},  \tilde{\nabla}) = (\psi_1)^*(\cL',  \nabla').
$$
\end{enumerate} 

We say $\hat{\psi}= (\psi,  \tilde{\alpha}, \tilde{\cL},  \tilde{\nabla})$ is an isotopy from $\hat{L}$ to $\hat{L}'$. 
We say the isotopy $\hat{\psi} =(\psi,\tilde{\alpha}, \tilde{\cL}, \tilde{\nabla})$ is {\em exact} if $\psi$ is an exact Lagrangian isotopy. 
\end{definition}

Define
\begin{equation}\label{eq:1form_defn}
\theta_\psi := \int_0^1 b_t  dt \in \Omega^1(L)
\end{equation}
which is a real closed 1-form on $L$ determined by $\psi$. 
\begin{lemma}\label{lem:exact1form} If $\psi$ is an exact Lagrangian isotopy then 
$\theta_\psi$  is an exact 1-form. 
\end{lemma}
\begin{proof}
If $\psi$ is exact then $b_t = d H_t$ where $H_t: L\to \bR$ is smooth. Define
$$
f :=\int_0^1 H_t dt \ \in \Omega^0(L).
$$
Then $df =\theta_\psi \in \Omega^1(L)$.
\end{proof}

Definition \ref{smooth-disks_1Lagr}, Definition \ref{smooth-homotopy_1Lagr}, and Lemma \ref{rho_1Lagr} below are analogous to Definition \ref{smooth-disks}, Definition \ref{smooth-homotopy}, and Lemma \ref{rho} but with only one Lagrangian and no marked point. 

\begin{definition}\label{smooth-disks_1Lagr}
Given a Lagrangian submanifold $L$ in $(X,\omega)$, let $\cD(X,L)$ be the set of continuous maps of pairs
$$
u: (\bD, \partial \bD)\lra (X,L)
$$
such that 
\begin{itemize}
\item $u: \bD\to X$ is a piecewise smooth continuous map, and
\item $\partial u:= u|_{\partial\bD} :\partial \bD \simeq S^1 \to L$ is a piecewise smooth loop in $L$.
\end{itemize} 
\end{definition}

\begin{definition} \label{smooth-homotopy_1Lagr}
A {\em homotopy} between two maps $u, u'$ in $\cD(X, L)$ is a continuous map of pairs
$$
h:  (\bD, \partial \bD) \times[0,1]  \lra (X, L )  
$$
satisfying the following conditions. 
\begin{enumerate}
\item[H1.] For any $t\in [0,1]$, define  $h_t:\bD\lra X$ by $h_t(z)= h(z,t)$.  Then
$h_t$ is a map in $\cD(X,L)$, $h_0=u$, and $h_1=u'$. 
\item[H2.] $h$ is a piecewise smooth continuous map.
\item[H3.]  The map $\partial  h= h|_{\partial \bD \times [0,1]}: \partial  \bD \times [0,1] \simeq S^1\times [0,1] \to L$ is piecewise smooth.   
\end{enumerate} 
We say two maps in $\cD(X,L)$ are  {\em homotopic} if there exists a homotopy between them.
\end{definition} 

Note that in H3, $\partial h$ is a  homotopy between the two loops $\partial u$ and $\partial u'$ in $L$. 
\begin{definition}
Given an object $\hat{L}=(L, \widetilde{\alpha_L}, \cL, \nabla)$ in $\Fuk(\uX)$,  we define  $\rho: \cD(X,L)  \to \bC^*$ by 
$$
\rho(u) = e^{2\pi i\int_{\bD}u^*\omega_\bC } \Hol_{\nabla}(\partial u). 
$$
\end{definition} 

\begin{lemma}\label{rho_1Lagr} If 
$u, u' \in \cD(X, L)$ are  homotopic  then  $\rho(u)=\rho(u') \in \bC^*$. 
\end{lemma}
\begin{proof} Let $h:\bD\times[0,1]\to X$ be a homotopy between $u$ and $u'$. 
$L$ is a Lagrangian submanifold of $(X,\omega)$,  so $\omega|_L=0$, which implies 
\begin{equation}
\int_{\partial  \bD \times [0,1]}  (\partial h)^* (\omega|_L) =0
\end{equation} 
The curvature $F_{\nabla}$  of  $\nabla$ satisfies $F_{\nabla} = -2\pi i B|_{L}$, so  
\begin{equation} \label{eqn:Hol-Hol-single} 
\Hol_{\nabla}(\partial  u') \circ \Hol_{\nabla }(\partial  u)^{-1} =  \exp\left(2\pi i\int_{ \partial  \bD \times [0,1]} (\partial  h)^*(B|_L )\right). 
\end{equation} 

Note that $d h^*\omega = h^*d\omega =0$ and $d h^*B = h^*dB=0$. By Stokes' theorem,
\begin{equation} \label{eqn:int-omega-single}
0  =  \int_{\bD\times [0,1]} d (h^*\omega) = \int_{\bD} h_1^*\omega    - \int_{\bD} h_0^*\omega + \int_{\partial  \bD\times [0,1]} (\partial  h)^*(\omega|_{L}) 
= \int_{\bD} (u')^*\omega  -\int_{\bD} u^*\omega 
\end{equation}
\begin{equation} \label{eqn:int-B-single} 
\begin{aligned}
0  =&  \int_{\bD \times [0,1]} d (h^*B) = \int_{\bD} h_1^*B    - \int_{\bD} h_0^*B + \int_{\partial  \bD\times [0,1]} (\partial  h)^*(B|_{L}) \\
=& \int_{\bD} (u')^*B  -\int_{\bD} u^*B +  \int_{\partial \bD \times [0,1]}(\partial  h)^*(B|_{L}). 
\end{aligned} 
\end{equation} 
\begin{eqnarray*}
\rho(u') &=&  e^{ 2\pi i  \int_{\bD}{u'}^*\omega_\bC}\Hol_{\nabla'} (\partial  u') 
=   e^{ 2\pi i  \int_{\bD}u^*\omega_{\bC}} \exp\left(-2\pi i\int_{\partial  \bD \times [0,1]} (\partial h)^*(B|_{L}) \right) \Hol_{\nabla} (\partial  u')  \\
&=&  e^{ 2\pi i  \int_{\bD}u^*\omega_{\bC}} \Hol_{\nabla} (\partial  u) 
= \rho(u) 
\end{eqnarray*}
where the second equality follows from  \eqref{eqn:int-omega-single} and \eqref{eqn:int-B-single}, while the third equality follows from \eqref{eqn:Hol-Hol-single}. 

\end{proof}

The following proposition tells us how $\rho(u) \in \bC^*$ changes when we deform $\hL$ by a (possibly non-exact) isotopy.

Given an isotopy $\hat{\psi}$ be between two objects $\hat{L}$ and $\hat{L}'$ in $\Fuk(\uX)$, and define
a map $\cD(X,L)\to \cD(X,L')$ by sending $u$ to $u'$ where
\begin{equation}\label{eqn:define u'}
u'(r e^{i \theta}) = \begin{cases}
u(2r e^{i\theta}), & 0\leq r\leq \frac{1}{2},\\
\psi^{2r-1}\circ u(e^{i\theta}), & \frac{1}{2} \leq r\leq 1.
\end{cases}
\end{equation}

\begin{proposition}\label{prop:Lag_isotop}
$$
\rho(u')  = \rho(u) e^{2\pi \int_{\partial \bD} (\partial u)^* \theta_\psi}= \rho(u) 
e^{2\pi \langle [\theta_\psi], [\partial u] \rangle }
$$
where $[\theta_\psi ] \in H^1(L;\bR)$ and $[\partial u] \in H_1(L;\bZ)$ is in the kernel of $H_1(L;\bZ)\to H_1(X;\bZ)$. In particular, if $\psi$ is exact then $[\theta_\psi] =0\in H^1(L;\bR)$, so 
$\rho(u')= \rho(u)$.
\end{proposition}
Note that the factor $e^{2\pi \int_{\partial \bD} (\partial u)^* \theta_\psi}$ only depends on $\psi$ and $\omega$; it is independent of $B$. 
\begin{proof}
We have a map $(\partial u)\times \mathrm{id}: \partial\bD\times [0,1]\lra L\times [0,1]$
given by $(e^{i\theta}, t)\mapsto (u(e^{i\theta}), t)$.  Then 
\begin{equation}\label{eqn:integral-annulus}
\int_{\bD} (u')^*\omega_\bC 
=\int_{\bD} u^*\omega_\bC
-\int_{\partial \bD \times [0,1]} (\partial u \times \mathrm{id})^*\psi^*\omega_\bC
\end{equation}
We get a negative sign on the right hand side of \eqref{eqn:integral-annulus} because the orientation on $\bD$ is compatible with $dr\wedge d\theta$, while the orientation on $\partial \bD\times [0,1]$ is compatible with $d\theta\wedge dr = -dr\wedge d\theta$. 
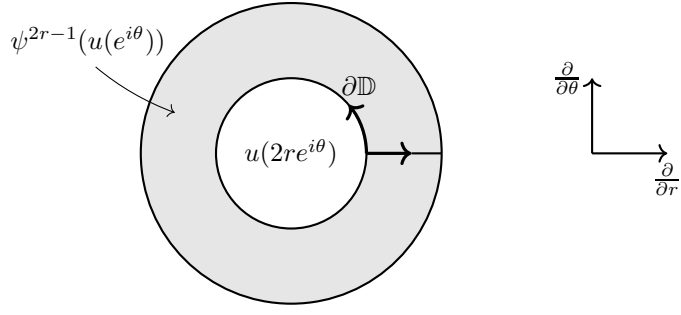
\begin{figure}
\centering                    \begin{tikzpicture}
    \draw [thick, draw,fill=gray!20](0,0) circle (2cm); 
     \draw [thick, draw,fill=white](0,0) circle (1cm); 
     \draw[very thick, ->] (1,0) arc (0:40:1);
     \draw[very thick, ->] (1,0)--(1.6,0); 
     \draw[thick] (1.6,0)--(2,0);

    \node at (0,0) {$u(2r e^{i\theta})$};
    \node at (-2.7, 1.5){$\psi^{2r-1}(u(e^{i\theta}))$};
    \draw[->] (-2.6, 1.15) arc (230:250:3.5);
    \node at (0.9, 0.9) {$\partial \mathbb D$};

    \draw[thick, ->] (4,0)--(5,0) node[below] {$\frac{\partial}{\partial r}$};
    \draw[thick, ->] (4,0)--(4,1) node[left] {$\frac{\partial}{\partial \theta}$};
     \end{tikzpicture}
    \caption{The map $u'$  defined in Equation \eqref{eqn:define u'}. The orientation on $\partial \bD\times [0,1]$ is given by the ordered basis $(e_1=\frac{\partial}{\partial \theta}, e_2=\frac{\partial}{\partial r})$.}
    \label{fig:annulus}
\end{figure}
We also have
\begin{equation}\label{eqn:integrand-annulus}
(\partial u \times \mathrm{id})^* \psi^* \omega_{\bC} 
=(\partial u\times \mathrm{id})^*( \psi^*B + i\psi^*\omega) 
= (\partial u\times \mathrm{id})^*(\frac{i}{2\pi} F_{\tilde{\nabla}}) + i (\partial u\times \mathrm{id})^*(b_t\wedge dt)             
\end{equation}
Equation \eqref{eqn:integral-annulus} and Equation \eqref{eqn:integrand-annulus} imply 
\begin{equation}\label{eq:area_only_change}
2\pi i \int_{\bD}(u')^* \omega_\bC = 2\pi i \int_{\bD} u^*\omega_\bC
+ \int_{\partial \bD \times [0,1]}  (\partial u\times \mathrm{id})^* F_{\tilde{\nabla}} + 2\pi \int_{\partial \bD} \theta_\psi.
\end{equation}
We also have
$$
\Hol_{\nabla'}(\partial u')\circ \Hol_\nabla (\partial u)^{-1} = 
\exp\left( -\int_{\partial \bD\times [0,1]} (\partial u\times \mathrm{id})^*F_{\tilde{\nabla}} \right).
$$
Therefore, 
\begin{eqnarray*}
\rho(u') &=& e^{2\pi i\int_{\bD}(u')^* \omega_\bC} \Hol_{\nabla'}(\partial \psi(u)) \\
&=&  e^{2\pi i \int_{\bD}u^*\omega_\bC}  \exp\left(\int_{\partial \bD\times [0,1]}(\partial u\times \mathrm{id})^*F_{\tilde{\nabla}}\right) e^{2\pi\int_{\partial\bD} \theta_\psi}
\Hol_{\nabla'}(\partial u') \\
&=& e^{2\pi i \int_{\bD}u^*\omega_\bC} \Hol_{\nabla}(\partial u) e^{2\pi\int_{\partial\bD} \theta_\psi }\\
&=& \rho(u) e^{2\pi\int_{\partial\bD } \theta_\psi }
\end{eqnarray*}
\end{proof}

Let $\hL_j$ and $p_j\in L_{j-1}\cap L_j$  be as in Section \ref{subsec:products}. We would like to understand how 
\begin{equation}\label{eq: product in isotopy section}
\mu^k: CF^{i_k}(\hL_{k-1}, \hL_k)\otimes \cdots \otimes CF^{i_1}(\hL_0, \hL_1)
\lra CF^{2-k + \sum_{j=1}^k i_j}(\hL_0, \hL_k)
\end{equation}
change as we deform $\hL_j$ by an isotopy.
It suffices to understand this when we deform only one of them, say $\hL_m$. \textcolor{black}{We divide up the domain of a disc $u$, counted in $\mu^k$, into $R_1 \cup R_2 = \mathbb{D}$. We do this in order to piecewise-define a new curve $u'$ with the new boundary condition. On the region $R_1$, we map the original curve $u$. On the region $R_2$ we move with the Lagrangian isotopy.} Let $\hat{\psi}$ be an isotopy between $\hL_m$ and $\hL'_m$ such that the map
$\psi: L_m\times [0,1] \to X$ is transverse to $L_{m-1}$ and $L_{m+1}$. 
We further assume that, for all $t\in [0,1]$,
$$
\psi(p_m,t)\in L_{m-1}, \quad \psi(p_{m+1},t)\in L_{m+1}.
$$
Let $\gamma_m: [0,1]\to L_{m-1}$ and $\gamma_{m+1}:[0,1]\to L_{m+1}$ be defined by 
$$
\gamma_m(t) =\psi(p_m,t),\quad \gamma_{m+1}(t) = \psi(p_{m+1},t),
$$
and let $p'_m := \gamma_m(1) \in L_{m-1}\cap L'_m$, $p'_{m +1}:= 
\gamma_{m+1}(1) \in L'_m\cap L_{m+1}$. 
Then $\gamma_m$ (resp. $\gamma_{m+1}$) is a path in $L_{m-1}$ (resp. $L_{m+1}$)
from $p_m$ (resp. $p_{m+1}$) to $p'_m$ (resp. $p'_{m+1}$). See Figure \ref{fig:isotopy} for illustration.

We define 
$$
\cD(X, (L_0,\ldots, L_k), (p_0,\ldots, p_k))
\to \cD(X, (L_0, \ldots, L_{m-1}, L'_m, L_{m+1},\ldots, L_m), (p_0,\ldots, p_{m-1}, p'_m, p'_{m+1},
p_{m+2}, \ldots, p_k)
$$
by sending $u$ to $u'$, where $u'$ is defined in terms of 
$$
u: \Big(\bD, \partial \bD=\bigcup_{j=0}^k \partial_j\bD, \{z_0,\ldots, z_k\}\Big)
\lra \Big( X, \bigcup_{j=0}^k L_j, \{p_0, \ldots, p_k\}\Big)
$$
and $\psi: L_m\times [0,1]\to X$ as follows. Without loss of generality, we may assume $z_0=1$, and
$z_j= e^{i\theta_j}$, where $0=\theta_0< \theta_1 < \cdots < \theta_k <2\pi$. Let
$$
\theta_m' =\frac{2\theta_m +\theta_{m+1}}{3}, \quad \theta_{m+1}' = \frac{\theta_m + 2\theta_{m+1}}{3}.
$$
Then $\theta_m<\theta'_m < \theta'_{m+1}<\theta_{m+1}$. Let $z'_m =e^{i\theta'_m}$ and $z'_{m+1}= e^{i\theta'_{m+1}}$.
Let $R_2 \subset \bD$ be the closed region bounded by $\partial_m \bD$ and
the line segment $\overline{z_m z_{m+1}}$ connecting $z_m$ and $z_{m+1}$, and let $R_1$ be the closure of 
$\bD\setminus R_2$. Then $R_1\cup R_2 =\bD$ and $R_1\cap R_2 = \overline{z_m z_{m+1}}$.
Let $C$ be the closure of $\partial\bD\setminus \partial_m\bD$. Then 
$$
\partial \bD = C\cup \partial_m\bD, \quad \partial R_1 = C\cup \overline{z_m z_{m+1}}, 
\quad \partial R_2 = \overline{z_m z_{m+1}} \cup \partial_m \bD. 
$$
\begin{figure}
\centering
\begin{subfigure}[b]{0.45\textwidth}
\begin{tikzpicture}
\draw[thick, red, fill=gray!20] (-1, {sqrt(3)}) arc (120:240:2)node[midway, left]{$\partial_m\mathbb D$};
\draw[thick, dotted] (0,-2) arc (-90:-60:2);
\draw[thick] (1, -{sqrt(3)}) arc (-60:60:2);
\draw[thick, dotted] (1, {sqrt(3)}) arc (60:90:2);
\draw[thick] (0, 2) arc (90:120:2);
\draw[thick] (-1, {-sqrt(3)}) arc (-120:-90:2);
\draw (-1, {sqrt(3)})--(-1, {-sqrt(3)});

\filldraw[black] (2,0) circle (2pt) node[right]{$z_0$};
\filldraw[black] ({sqrt(3)},1) circle (2pt) node[right]{$z_1$};
\filldraw[black] ({sqrt(3)},-1) circle (2pt) node[right]{$z_k$};
\filldraw[Cblue] (-1, {sqrt(3)}) circle (2pt) node[left]{$z_m$};
\filldraw[Cblue] (-1.879, 0.684) circle (2pt) node[left]{$z_m'$};
\filldraw[Cgreen] (-1, {-sqrt(3)}) circle (2pt) node[left]{$z_{m+1}$};
\filldraw[Cgreen] (-1.879, -0.684) circle (2pt) node[left]{$z_{m+1}'$};
\node at (2.3, 0.5) {$\partial_0\mathbb D$};
\node at (2.3, -0.5) {$\partial_k\mathbb D$};

\node at (-1.5, 0) {$R_2$};
\node at (0.5,0){$R_1$};
\end{tikzpicture}
\end{subfigure}
\begin{subfigure}[b]{0.45\textwidth}

\begin{tikzpicture}

\draw[thick] (1.83, 1.5) .. controls +(-0.02, -0.02) and +(0.1, 0.5).. ({sqrt 3}, 1) .. controls +(-0.1,-0.5) and +(-0.1, 0.2) .. (2,0)node[midway, left]{$L_{0}$}.. controls +(0.1, -0.2) and +(-0.02,0.02) .. (2.2,-0.4);

\draw[thick] (1.83, -1.6).. controls +(-0.02, 0.02) and +(0.1, -0.6).. ({sqrt 3}, -1).. controls +(-0.1, 0.6) and +(-0.1, -0.2)..(2,0)node[midway,left]{$L_{k}$}.. controls +(0.1, 0.2) and +(-0.02, -0.02).. (2.2, 0.4);

\draw[thick] (-0.7, 2.4)node[right]{$L_{m-1}$}.. controls+(-0.02, -0.02) and+(0.2,0.4)..(-1, {sqrt(3)}).. controls +(-0.2,-0.4) and +(0.3,0.2).. (-1.879, 0.684).. controls +(-0.3,-0.2) and +(0.02,0.02)..(-2.5,0.3) ;

\draw[thick](-0.7, -2.4)node[right]{$L_{m+1}$}.. controls+ (-0.02,0.02) and+(0.2,-0.4)..(-1, {-sqrt(3)})..controls +(-0.2,0.4) and +(0.3,-0.2)..(-1.879, -0.684)..controls +(-0.3,0.2) and +(0.02, -0.02) ..(-2.5,-0.3);

\draw[very thick, Cblue] (-1, {sqrt(3)}).. controls +(-0.2,-0.4) and +(0.3,0.2).. (-1.879, 0.684)  node[midway, below, xshift=0.1cm, yshift=-0.1cm]{$\gamma_m$};
\draw[very thick, Cgreen] (-1, {-sqrt(3)})..controls +(-0.2,0.4) and +(0.3,-0.2)..(-1.879, -0.684) node[midway,above, xshift=0.2cm]{$\gamma_{m+1}$};

\draw[thick, Corange] (-2.1, 1.1).. controls+(0.02,-0.02) and+(-0.2,0.4)..(-1.879, 0.684)..controls +(0.2,-0.4) and+(0.2,0.4)..(-1.879, -0.684)node[midway, right]{$L_m'$}.. controls +(-0.2,-0.4) and+ (0.02,0.02).. (-2.1,-1.1);

\draw[thick, red](-1.45, 2.6).. controls +(0.05, -0.05) and +(-0.5,1).. (-1, {sqrt 3}).. controls +(0.5, -1) and +(0.5, 1) .. (-1,{-sqrt(3)})node[midway, right]{$L_m$}.. controls +(-0.5,-1) and +(0.05,0.05) .. (-1.45, -2.6) ;

\filldraw[black] (2,0) circle (2pt) node[right]{$p_0$};
\filldraw[black] ({sqrt(3)},1) circle (2pt) node[right]{$p_1$};
\filldraw[black] ({sqrt(3)},-1) circle (2pt) node[right]{$p_k$};
\filldraw[Cblue] (-1, {sqrt(3)}) circle (2pt) node[left]{$p_m$};
\filldraw[Cblue] (-1.879, 0.684) circle (2pt) node[left,yshift=0.1cm]{$p_m'$};
\filldraw[Cgreen] (-1, {-sqrt(3)}) circle (2pt) node[left]{$p_{m+1}$};
\filldraw[Cgreen] (-1.879, -0.684) circle (2pt) node[left]{$p_{m+1}'$};

\draw[thick, dotted] (0,-2) arc (-90:-60:2);
\draw[thick, dotted] (1, {sqrt(3)}) arc (60:90:2);
\end{tikzpicture}
\end{subfigure}
\caption{Domain (left) and target (right) of a holomorphic disc contributing to $\mu^k$.  There is an isotopy between Lagrangians $L_m$ and $L_m'$. This figure illustrates the notations in Section \ref{sec: isotopy}, below Equation \eqref{eq: product in isotopy section}.}
\label{fig:isotopy}
\end{figure}
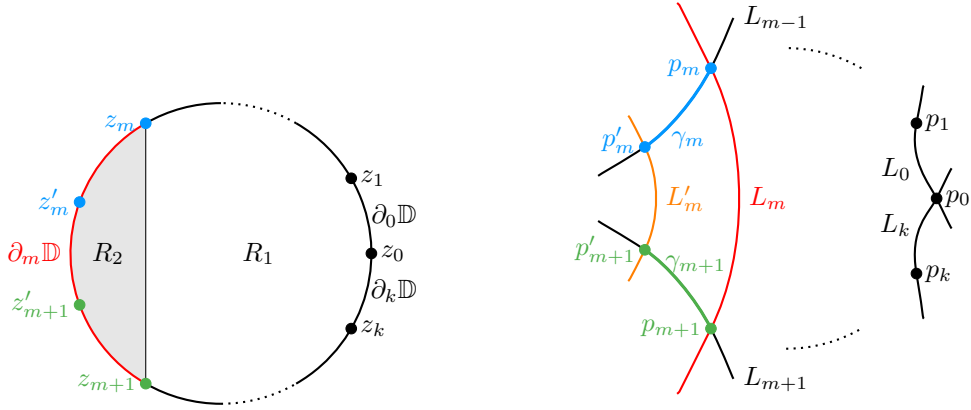
There is a homeomorphism $\phi_1: R_1\to \bD$ which satisfies the following properties.
\begin{itemize}
\item $\phi_1$ restricts to a diffeomorphism from the interior of $R_1$ to the interior of $\bD$.
\item $\phi_1|_C: C\to C$ is the identity map.
\item $\phi_1$ maps the interior of $\overline{z_m z_{m+1}}$ diffeomorphically
to the interior of $\partial_m \bD$. 
\end{itemize}
There is a homeomorphism $\phi_2: R_2\to \partial_m \bD\times [0,1]$ satisfying
the following properties:
\begin{itemize}
    \item  $\phi_2$ restricts to a diffeomorphism from the interior of $R_2$ to 
$\partial_m \bD\times [0,1]$.
\item $\phi_2(z)= (\phi_1(z),0)$ for $z\in \overline{z_m z_{m+1}}$. In particular, 
$\phi_2(z_m)=(z_m,0)$ and $\phi_2(z_{m+1})= (z_{m+1},0)$.
\item $\phi_2|_{\partial_m\bD}:\partial_m \bD \lra \{ z_m, z_{m+1}\}\times [0,1] \cup \partial_m\mathbb D\times \{1\} $  
is piecewise smooth
and $\phi_2(z'_m)=(z_m,1)$, $\phi_2(z'_{m+1})=(z_{m+1},1)$.
\end{itemize}
See Figure \ref{fig:psi2} for an illustration of the domain and image of $\phi_2$.
\begin{figure}
\centering
\begin{subfigure}[b]{0.45\textwidth}
\begin{tikzpicture}
\draw[thick, red, fill=gray!20] (-1, {sqrt(3)}) arc (120:240:2)node[midway, left]{$\partial_m\mathbb D$};
\draw[thick] (-1, {sqrt(3)})--(-1, {-sqrt(3)}) node[midway, right]{$\overline{z_mz_{m+1}}$};

\filldraw[Cblue] (-1, {sqrt(3)}) circle (2pt) node[left]{$z_m$};
\filldraw[Cblue] (-1.879, 0.684) circle (2pt) node[left]{$z_m'$};
\filldraw[Cgreen] (-1, {-sqrt(3)}) circle (2pt) node[left]{$z_{m+1}$};
\filldraw[Cgreen] (-1.879, -0.684) circle (2pt) node[left]{$z_{m+1}'$};

\node at (-1.5, 0) {$R_2$};
\end{tikzpicture}
\end{subfigure}
\begin{subfigure}[b]{0.45\textwidth}
\begin{tikzpicture}
\fill[fill=gray!20] (-1, {sqrt(3)}) arc (120:240:2)--(-1, {-sqrt(3)})--(0,{-sqrt(3)})--(0, {-sqrt(3)}) arc (240:120:2)--(0,{sqrt (3)})--(-1, {sqrt(3)});

\draw[thick] (-1, {sqrt(3)}) arc (120:240:2)node[midway, left]{$\phi_2(\overline{z_mz_{m+1}})$};
\draw[thick, red] (0, {sqrt(3)}) arc (120:240:2)node[midway, right]{$\phi_2(\partial_m\mathbb D)$};
\draw[thick, red](-1, {sqrt(3)})--(0,{sqrt(3)});
\draw[thick, red](-1, {-sqrt(3)})--(0,{-sqrt(3)});

\filldraw[Cblue] (-1, {sqrt(3)}) circle (2pt) node[left]{$\phi_2(z_m)=(z_m,0)$};
\filldraw[Cblue] (0, {sqrt(3)}) circle (2pt) node[right]{$\phi_2(z_m')=(z_m, 1)$};
\filldraw[Cgreen] (-1, {-sqrt(3)}) circle (2pt) node[left]{$\phi_2(z_{m+1})=(z_{m+1},0)$};
\filldraw[Cgreen] (0, {-sqrt(3)}) circle (2pt) node[right]{$\phi_2(z_{m+1}')=(z_{m+1},1)$};

\end{tikzpicture}
\end{subfigure}
\caption{The domain (left) and image (right) of $\phi_2:R_2\to \partial_m\mathbb D\times [0,1]$.}
\label{fig:psi2}
\end{figure}
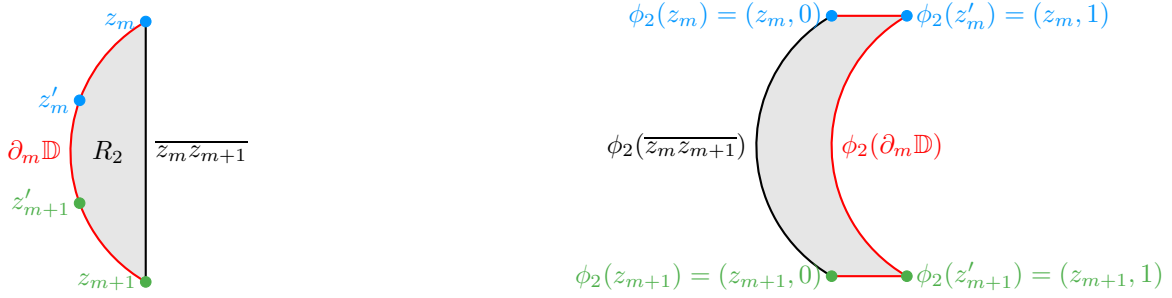

Define $P: \partial_m\bD \times [0,1]\to X$ by $P(z,t)= \psi_t \circ \partial_m u(z)$. 
Define $u':\bD\to X$ by 
\begin{equation}\label{eq:u'_mult_Lagr}
u'(z) = \begin{cases}
u\circ \phi_1(z), & z\in R_1;\\
P\circ \phi_2(z), & z\in R_2.
    \end{cases}
\end{equation}
Let $\partial'_{m-1}\bD$, $\partial'_m\bD$, $\partial'_{m+1}\bD$ denote
the arcs from $z_{m-1}$ to $z_m'$, $z_m'$ to $z'_{m+1}$, and  to $z_{m+1}'$ to  $z_{m+2}$, respectively.  
Then $u'$ defines a map from
$$
\Big(\bD, \partial \bD=\bigcup_{j=0}^{m-2} \partial_j\bD
\cup \bigcup_{j=m-1}^{m+1} \partial'_j\bD\cup \bigcup_{j=m+2}^k \partial_j\bD
, \{z_0,\ldots, z_{m-1}, z'_m, z'_{m+1}, z_{m+2}, \ldots,  z_k\}\Big)
$$
to
$$
\Big( X, \bigcup_{j=0}^{m-1} L_j \cup L'_m \cup \bigcup_{j=m+1}^k L_j, \{p_0, \ldots, p_{m-1}, p'_m,
p'_{m+1}, p_{m+2},\ldots, p_k\}\Big)
$$
and is an element in 
$$\cD(L_0,\ldots, L_{m-1},L'_m, L_{m+1}, \ldots, L_k),
(p_0, \ldots, p_{m-1}, p'_m, p'_{m+1}, p_{m+2},\ldots, p_k).
$$

Given
$$
\rho_j \in \Hom_{\bC}((\cL_{j-1})_{p_j}, (\cL_j)_{p_j}), \quad j=1,\ldots, k.
$$
Define
\begin{eqnarray*}
\rho'_m &:=&  \Hol_{\tilde{\nabla}}(\gamma_m)\circ \rho_m \circ \Hol_{\nabla_{m-1}}(\gamma_m)^{-1} \in \Hom_{\bC}((\cL_{m-1})_{p'_m}, (\cL'_m)_{p'_m}) \\
\rho'_{m+1} &:=& \Hol_{\nabla_{m+1}}(\gamma_{m+1})\circ  \rho_{m+1}\circ \Hol_{\tilde{\nabla}}(\gamma_{m+1})^{-1} \in \Hom_{\bC}((\cL'_m)_{p'_{m+1}}, (\cL_{m+1})_{p'_{m+1}})
\end{eqnarray*}

\begin{theorem}\label{prop:polygon-area}
$$
\rho(\rho_k, \cdots, \rho_{m+2}, \rho'_{m+1}, \rho'_m, \rho_{m-1}, \ldots, \rho_1; u') = 
\rho(\rho_k,\ldots, \rho_1; u) e^{2\pi\int_{\partial_m \bD} (\partial_m u)^*\theta_\psi }.
$$
\end{theorem}

\begin{remark}\label{rem:exact_integrl}
The 1-form $\theta_\psi$ is closed, so the integral 
$$
\int_{\partial_m \bD}(\partial_m u)^*\theta_\psi
$$
depends only on the fixed-end-point homotopy class of the path
$\partial_m u: \partial_m\bD \simeq [0,1] \to L_m$.

If $\theta_\psi = df$ is exact then 
$$
\int_{\partial_m \bD}(\partial_m u)^*\theta_\psi = f(p_{m+1}) - f(p_m)
$$
which is independent of $u$. 
\end{remark}

\begin{proof}[Proof of Theorem \ref{prop:polygon-area}] We have
\begin{equation} \label{eqn:integral-rectangle}
\int_{\bD} (u')^*\omega_\bC 
=\int_{\bD} u^*\omega_\bC
-\int_{\partial_m \bD \times [0,1]} (\partial_m u \times \mathrm{id})^*\psi^*\omega_\bC.
\end{equation}
where
\begin{equation} \label{eqn:integrand-rectangle}
(\partial_m u \times \mathrm{id})^* \psi^* \omega_{\bC} 
=(\partial_m u\times \mathrm{id})^*( \psi^*B + i\psi^*\omega) 
= (\partial_m u\times \mathrm{id})^*(\frac{i}{2\pi} F_{\tilde{\nabla}}) + i (\partial_m u\times \mathrm{id})^*(b_t\wedge dt).
\end{equation}
Equation \eqref{eqn:integral-rectangle} and Equation \eqref{eqn:integrand-rectangle} imply
\begin{equation}\label{eq:area_only_change_multLag}
2\pi i \int_{\bD}(u')^* \omega_\bC = 2\pi i \int_{\bD} u^*\omega_\bC
+ \int_{\partial_m \bD \times [0,1]}  (\partial_m u\times \mathrm{id})^* F_{\tilde{\nabla}} + 2\pi \int_{\partial_m \bD} \theta_\psi.
\end{equation}
We also have
\begin{eqnarray*}
\Hol_{\nabla_j}(\partial_j u') &=& \Hol_{\nabla_j}(\partial_j u) \quad\text{if } j\in \{0,1\ldots,k\}-\{m-1, m, m+1\},\\
\Hol_{\nabla_{m-1}}(\partial_{m-1} u') &=& \Hol_{\nabla_{m-1}}(\gamma_m) \circ \Hol_{\nabla_{m-1}}(\partial_{m-1} u),\\
\Hol_{\nabla'_m}(\partial_m u') &=& \Hol_{\tilde{\nabla}}(\gamma_{m+1})
\circ \Hol_{\nabla_m}(\partial_m u) \circ \Hol_{\tilde{\nabla}}(\gamma_m)^{-1} 
\exp\left(- \int_{\partial_m\bD\times [0,1]} (\partial_m u\times \mathrm{id})^*F_{\tilde{\nabla}}\right), \\
\Hol_{\nabla_{m+1}} (\partial_{m+1}u') &=& \Hol_{\nabla_{m+1}}(\partial_{m+1}u) \circ \Hol_{\nabla_{m+1}}(\gamma_{m+1})^{-1}.
\end{eqnarray*}
Note that $\Hol_{\nabla_{m-1}}(\gamma_m)=\Hol_{\tilde \nabla}(\gamma_m)$ and $\Hol_{\nabla_{m+1}}(\gamma_{m+1})=\Hol_{\tilde \nabla}(\gamma_{m+1})$.
Therefore, 
\begin{eqnarray*}
&& \rho(\rho_k,\ldots, \rho_{m+2}, \rho'_{m+1}, \rho'_m, \rho_{m-1},\ldots,\rho_1; u') \\ 
&=& e^{2\pi i\int_{\bD}(u')^* \omega_\bC} \Hol_{\nabla_k}(\partial_k u') \circ \rho_k 
\circ\Hol_{\nabla_{k-1}}(\partial_{k-1}u') \circ \rho_{k-1} \circ ... \circ \rho_{m+2} \\
&& \circ
\Hol_{\nabla_{m+1}}(\partial_{m+1}u')\circ \rho'_{m+1} \circ \Hol_{\nabla'_m}(\partial_m u') \circ \rho'_m \circ
\Hol_{\nabla_{m-1}}(\partial_{m-1}u') \circ \rho_{m-1} \circ \cdots \circ \rho_1 \circ\Hol_{\nabla_0}(\partial_0 u')\\ 
&=& e^{2\pi i\int_{\bD}(u')^* \omega_\bC} \exp\left(-\int_{\partial_m \bD\times [0,1]}(\partial_m u\times \mathrm{id})^*F_{\tilde{\nabla}}\right) \Hol_{\nabla_k}(\partial_k u) \circ \rho_k 
\circ\Hol_{\nabla_{k-1}}(\partial_{k-1}u) \circ \rho_{k-1} \circ ... \circ \rho_{m+2} \\
&& \circ
\Hol_{\nabla_{m+1}}(\partial_{m+1}u)\circ \rho_{m+1} \circ \Hol_{\nabla_m}(\partial_m u) \circ \rho_m \circ 
\Hol_{\nabla_{m-1}}(\partial_{m-1}u) \circ \rho_{m-1} \circ \cdots \circ \rho_1 \circ\Hol_{\nabla_0}(\partial_0 u) \\ 
&=& e^{2\pi i\int_{\bD}u ^* \omega_\bC} e^{2\pi\int_{\partial_m\bD}\theta_\psi} \Hol_{\nabla_k}(\partial_k u) \circ \rho_k 
\circ\Hol_{\nabla_{k-1}}(\partial_{k-1}u) \circ \rho_{k-1} \circ ... \circ \rho_{m+2} \\
&& \circ
\Hol_{\nabla_{m+1}}(\partial_{m+1}u)\circ \rho_{m+1} \circ \Hol_{\nabla_m}(\partial_m u) \circ \rho_m \circ
\Hol_{\nabla_{m-1}}(\partial_{m-1}u) \circ \rho_{m-1} \circ \cdots \circ \rho_1 \circ\Hol_{\nabla_0}(\partial_0 u) \\ 
&=& \rho(\rho_k,\ldots, \rho_1; u) e^{2\pi\int_{\partial_m \bD } \theta_\psi }
\end{eqnarray*}
\end{proof}

\color{black}

Following the notation depicted in Figure \ref{fig:psi2}, we set up the domain so that the new $u'$ (obtained by isotoping one Lagrangian from the original map $u$) may be defined. Let $u_t$ denote $u'$ when $L_m'$ is taken to be $\psi(L_m,t)$, and similarly define $p_t$. As a corollary of Theorem \ref{prop:polygon-area}, we have the following main result of our paper. 

\begin{theorem}[How structure maps change under Lagrangian isotopy]\label{thm:str_map_isotopy}  We assume the conditions of Remark \ref{rem:convergence}. Let $\hat{\psi}$ be an isotopy between $\hL_m$ and $\hL'_m$ such that the map
$\psi: L_m\times [0,1] \to X$ is transverse to $L_{m-1}$ and $L_{m+1}$ and for all $t\in [0,1]$, we have $\psi(p_m,t)\in L_{m-1}, \; \psi(p_{m+1},t)\in L_{m+1}.$ We also assume that $\cup_{t \in [0,1]} \cM^0(\uX, (\hL_0,\ldots,\psi(\hL_m,t),\ldots, \hL_k), (p_0,\ldots, p_k);[u_t])$ is compact, where $\cM^0$ denotes the 0-dimensional part of $\cM$.

Then if $ \theta_\psi = df$ for $\theta_\psi$ defined in Equation \eqref{eq:1form_defn}, the structure maps are related by
\begin{equation} 
\mu^k (\rho_k,\ldots,\rho_m',\ldots, \rho_1) 
=e^{2\pi (f(p_{m+1}) - f(p_m))}\mu^k (\rho_k,\ldots, \rho_1).
\end{equation}

\end{theorem}

\begin{proof}
    Recall that
    \begin{equation} 
\mu^k (\rho_k,\ldots, \rho_1) 
=\sum_{\substack{ p_0\in L_0\cap L_k \\ \deg(\tL_0, \tL_k; p_0) 
=2-k+\sum_{j=1}^k i_j \\ [u] } }
\sharp \cM(\uX, (\hL_0,\ldots, \hL_k), (p_0,\ldots, p_k);[u])
\rho(\rho_k,\ldots, \rho_1;[u]).
\end{equation}
By Remark \ref{rem:exact_integrl} and Theorem \ref{prop:polygon-area}, 
$$
\rho(\rho_k,\ldots,\rho_m',\ldots \rho_1;[u'])=e^{2\pi (f(p_{m+1}) - f(p_m))}\rho(\rho_k,\ldots, \rho_1;[u]).
$$
To relate the moduli spaces, we take the 1-dimensional family of moduli spaces obtained by the Lagrangian isotopy 
$$
\cup_{t \in [0,1]} \cM(\uX, (\hL_0,\ldots,\psi(\hL_m,t),\ldots, \hL_k), (p_0,\ldots, p_k);[u_t]).
$$
Then the signed count of the boundary of a compact 1-dimensional manifold is 0, so the count of moduli spaces at either end is equal. In other words, 
$$
\cM(\uX, (\hL_0,\ldots,\hL_m,\ldots, \hL_k), (p_0,\ldots, p_k);[u])
$$ 
is equal to 
$$
\cM(\uX, (\hL_0,\ldots,\hL_m',\ldots, \hL_k), (p_0,\ldots,p_{m-1}',p_m',p_{m+1}',\ldots p_k);[u'])
$$ 
and 
\begin{equation} 
\mu^k (\rho_k,\ldots,\rho_m',\ldots, \rho_1) 
=e^{2\pi (f(p_{m+1}) - f(p_m))}\mu^k (\rho_k,\ldots, \rho_1).
\end{equation}
\end{proof}

\color{black}

\bibliographystyle{amsalpha}
\bibliography{glob_gen2_hms}

\end{document}